\documentclass{amsart} 
\usepackage{amsthm, amsfonts, amssymb, amsmath} 
\newtheorem{theorem}{Theorem}[section]
\newtheorem{lemma}[theorem]{Lemma}
\newtheorem{corollary}[theorem]{Corollary}
\newtheorem{proposition}[theorem]{Proposition}
\newtheorem{remark}[theorem]{Remark}

\theoremstyle{definition}
\newtheorem{definition}[theorem]{Definition}
\newtheorem{example}[theorem]{Example}
\newtheorem{remark/example}[theorem]{Remark/Example}

 \let\oldlabel=\label
\def\prellabel{\marginparsep=1em\marginparwidth=44pt
 \def\label##1{\oldlabel{##1}\ifmmode\else\ifinner\else
 \marginpar{{\footnotesize\ \\ \tt
 ##1}}\fi\fi}}

\numberwithin{equation}{section}

\def\PP{ {\bf P} }
\def\NN{ {\bf N} }
\def\ZZ{ {\bf Z} }
\def\QQ{ {\bf Q} }
\def\RR{ {\bf R} }

\def\aa{{ a}}
\def\cc{{\bf c}}
\def\ii{{\bf i}}
\def\jj{{\bf j}}
\def\id{id}
\def\ol#1{\overline{#1}}
 
\newcommand{\Rees}{\operatorname{Rees}}
\newcommand{\ini}{\operatorname{in}}
\newcommand{\GCD}{\operatorname{GCD}}

\newcommand{\mm}{\operatorname{{\mathbf m}}}

\newcommand{\depth}{\operatorname{depth}}

\newcommand{\Tor}{\operatorname{Tor}}

\newcommand{\reg}{\operatorname{reg}}
\newcommand{\gr}{\operatorname{gr}}
\newcommand{\hull}{\operatorname{New}}
\newcommand{\Syz}{\operatorname{Syz}}

\numberwithin{equation}{section}

\begin{document}

\title{Contracted ideals and the Gr\"obner fan of the rational normal curve
}
\author{ Aldo Conca, Emanuela De Negri, Maria Evelina Rossi}
\address{Dipartimento di Matematica, Universit\'a di Genova, Via Dodecaneso 35, I-16146 Genova, Italy }
\email{conca@dima.unige.it, denegri@dima.unige.it, rossim@dima.unige.it }
\date{}

\begin{abstract} 
The paper has two goals: the study of the associated graded ring of contracted homogeneous ideals in $K[x,y]$ and the study of the Gr\"obner fan of the ideal $P$ of the rational normal curve in $\PP^d$. These two problems are, quite surprisingly, very tightly related. We completely classify the contracted ideals with a Cohen-Macaulay associated graded rings in terms of the numerical invariants arising from Zariski's factorization. We determine explicitly the initial ideals (monomial or not) of $P$ that are Cohen-Macaulay. 
\end{abstract}
 
\maketitle
 \section{Introduction} 

The goal of the paper is twofold: 
\begin{itemize} 
\item[(1)] to describe the Cohen-Macaulay initial ideals of the defining ideal $P$ of the rational normal curve in $\PP^d$ in its standard coordinate system and for every positive integer $d$. 
\item[(2)] to identify the homogeneous contracted ideals in $K[x,y]$ whose associated graded ring is Cohen-Macaulay. 
\end{itemize} 

The two problems are closely related. Indeed they are essentially equivalent as we explain in the following. Let $K$ be a field, $R=K[x,y]$ and $I$ be a homogeneous ideal of $R$ with $\sqrt I=\mm=(x,y)$. Denote by $\gr_I(R)$ the associated graded ring $\oplus_k I^k/I^{k+1}$ of $I$. The ideal $I$ is said to be contracted if it is contracted from a quadratic extension, that is, if there exists a linear form $z$ in $R$ such that $I=IR[\mm/z]\cap R$. Contracted ideals have been introduced by Zariski in his studies on the factorization property of integrally closed ideals, see \cite[App.5]{ZS} or \cite[Chap.14]{HS}. Every integrally closed ideal $I$ is contracted and has a Cohen-Macaulay associated graded ring $\gr_I(R)$, see \cite{LT,H}. On the contrary, the associated graded ring of a contracted ideal need not be Cohen-Macaulay. Zariski proved a factorization theorem for contracted ideals asserting that every contracted ideal $I$ has a factorization $I=L_1\cdots L_s$ where the $L_i$ are themselves contracted but of a very special kind. In the homogeneous case and assuming $K$ is algebraically closed, each $L_i$ is a lex-segment monomial ideal in a specific system of coordinates depending on $i$. Recall that a monomial ideal $L$ in $R$ is a lex-segment ideal if whenever $x^ay^b\in L$ with $b>0$ then also $x^{a+1}y^{b-1}\in L$. In \cite[Corollary 3.14]{CDJR} it is shown that the Cohen-Macaulayness of $\gr_I(R)$ is equivalent to the Cohen-Macaulayness of $\gr_{L_i}(R)$ for every $i=1,\dots,s$. Therefore to answer (2) one has to characterize the lex-segment ideals $L$ with Cohen-Macaulay associated graded ring. Any lex-segment ideal $L$ of initial degree $d$ can be encoded by a vector $\aa=(a_0,a_1,\dots,a_d)$ with increasing integral coordinates and $a_0=0$. Given $L$ associated to $\aa$, we show that $\gr_L(R)$ is Cohen-Macaulay if and only if $\ini_\aa(P)$ defines a Cohen-Macaulay ring. Here $\ini_\aa(P)$ denotes the ideal of the initial forms of $P$ with respect to the vector $\aa$. Therefore (1) and (2) are indeed equivalent problems. In Section \ref{CM for P} we solve problem (1) by showing first that $P$ has exactly $2^{d-1}$ Cohen-Macaulay initial monomial ideals, see Theorem \ref{riassunto}. Then we show that every Cohen-Macaulay initial ideal of $P$ has a Cohen-Macaulay monomial initial ideal, Theorem \ref{nonmon}. In terms of the Gr\"obner fan of $P$, Theorem \ref{nonmon} can be rephrased by saying that $\ini_\aa(P)$ is Cohen-Macaulay iff $\aa$ belongs to the union of $2^{d-1}$ maximal closed cones. These cones are explicitly described by linear homogeneous inequalities. The fact that $P$ has exactly $2^{d-1}$ Cohen-Macaulay monomial initial ideals can be derived by combining results of Hosten and Thomas \cite{HT} with results of O'Shea and Thomas \cite{OT}, see Remark \ref{Rekha}. In Section \ref{contraCM} we give an explicit characterization, in terms of the numerical invariants arising from Zariski's factorization, of the Cohen-Macaulay property of the associated graded ring to a contracted homogeneous ideal in $K[x,y]$. In Section \ref{HF} we describe the relationship between the Hilbert series of $\gr_L(R)$ and the multigraded Hilbert series of $\ini_\aa(P)$. We discuss also how the formulas for the Hilbert series and polynomials of $\gr_L(R)$ change by varying the corresponding cones of the Gr\"obner fan of $P$. This has a conjectural relation with the hypergeometric Gr\"obner fan introduced by Saito, Sturmfels and Takayama in \cite{SST}. In Section \ref{bigcone} we show that the union of a certain subfamily of the $2^{d-1}$ Cohen-Macaulay cones is itself a cone. We call it the big Cohen-Macaulay cone. Indeed, the big Cohen-Macaulay cone is the union of $f_{d}$ Cohen-Macaulay cones of the Gr\"obner fan of $P$, where $f_{d}$ denotes the $(d+1)$-th Fibonacci number. In Section \ref{smalld} we present some examples. 

We thank J\"urgen Herzog, Rekha Thomas, Greg Smith and Bernd Sturmfels for useful discussions and references related to the topic of the present paper. 
Most of the results of this paper have been discovered and confirmed by extensive CoCoA \cite{CoCoA} computations.

\section{Notation and preliminaries} 
\label{notpre} 
Let $S$ be a polynomial ring over a field $K$ with maximal homogeneous ideal $\mm$. 
For a homogeneous ideal $I$ of $S$ we denote by $\gr_I(S), \Rees(I)$ and $F(I)$ respectively the associated graded ring $\oplus_{k\in \NN} I^k/I^{k+1}$ , the Rees algebra $\oplus_{k\in \NN} I^k$ and the fiber cone $\oplus_{k\in \NN} I^k/\mm I^{k}$ of $I$. By the very definition $F(I)$ is a standard graded $K$-algebra. Furthermore $\Rees(I)$ can be identified with the $S$-subalgebra of the polynomial ring $S[t]$ generated by $ft$ with $f\in I$. 
 
Let $I\subset S=K[x_1,\dots, x_n]$ be a homogeneous ideal. We may consider the (standard) Hilbert function, Hilbert polynomial and Hilbert series of $S/I$. The Hilbert series of $S/I$ is $\sum_{i\geq 0} \dim_K (S/I)_i z^i$ and we denote it by $H_{S/I}(z)$. The series $H_{S/I}(z)$ has a rational expression $h(z)/(1-z)^{d}$ where $h(z) \in \ZZ[z]$ and $d$ is the Krull dimension of $S/I$. The polynomial $h(z) $ is called the (standard) $h$-polynomial of $S/I$. In particular, $h(0)=1$ and $h(1)$ is the ordinary multiplicity of $S/I$ denoted by $e(S/I)$.

If $I$ is $\mm$-primary, we may consider also the (local) Hilbert functions, Hilbert polynomials and Hilbert series of $I$ (or of $\gr_I(S)$). 
There are two Hilbert functions associated with $I$ in this context. We denote them by $H(I,k)$ and $H^1(I,k)$ and are defined by 
$$H(I,k)= \dim_K (I^k/I^{k+1}) \quad \mbox{ and } \quad H^1(I,k)=\dim_K (S/I^{k+1}).$$
The corresponding Hilbert series are: 
$$H_I(z)=\sum_{k\geq 0} H(I,k) z^k \quad \mbox{ and } \quad H^1_I(z)=\sum_{k\geq 0} H^1(I,k) z^k.$$ 
Obviously, $H_I(z)=(1-z)H^1_I(z)$. The series $H^1_I(z)$ has a rational expression 
$$H^1_I(z)=\frac{h(z)}{(1-z)^{n+1}}$$
where $h(z)$ is a polynomial with integral coefficients called the (local) $h$-polynomial of $I$ or of $\gr_I(S)$. The Hilbert functions $H(I,k)$ and $H^1(I,k)$ agree for large $k$ with polynomials $P_I(z)$ and $P^1_I(z)$ at $z=k$. The polynomials $P_I(z)$ and $P^1_I(z)$ are called the Hilbert polynomials of $I$. Their coefficients, with respect to an appropriate binomial basis, are integers called Hilbert coefficients of $I$ and are denoted by $e_i(I)$. Precisely, 
$$P^1_I(z)=\sum_{i=0}^n (-1)^i e_{i}(I) \binom{n-i+z}{n-i}.$$
 In particular, $h(0)=\dim_K S/I$ and $h(1)=e_0(I)$ is the multiplicity of $I$. 

\begin{definition} Let $I\subset S=K[x_1,\dots, x_n]$ be a homogeneous ideal of codimension $c$ and not containing linear forms. 
Then: 
\begin{itemize}
\item[(1)] $S/I$ has minimal multiplicity if $e(S/I)=c+1$. 
\item[(2)] $S/I$ has a short $h$-vector if its $h$-polynomial is $1+cz $, that is, if the Hilbert series of $S/I$ is $(1+cz)/(1-z)^{n-c}$. 
\end{itemize} 
\end{definition} 

We denote by $\reg(M)$ the Castelnuovo-Mumford regularity of a graded $S$-module $M$. For results on the Castelnuovo-Mumford regularity and minimal multiplicity we refer the reader to \cite{EG}.
We just recall that if $S/I$ has a short $h$-vector, then it has minimal multiplicity. On the other hand, if $S/I$ is Cohen-Macaulay with minimal multiplicity, then it has a short $h$-vector. 
We will need the next lemma whose easy proof follows from standard facts. 

\begin{lemma}
\label{CMvsCM} Let $I\subset S$ be a homogeneous ideal. Assume $S/I$ has a short $h$-vector. Then $S/I$ is Cohen-Macaulay iff $\reg(I)=2$. \end{lemma}

Every vector $\aa=(a_0,\dots,a_d)\in \QQ_{\geq 0}^{d+1}$ induces a graded structure on the polynomial ring $K[t_0,\dots,t_d]$ by putting $\deg t_i=a_i$. Every monomial $t^\alpha$ is then homogeneous of degree 

$$\deg_{\aa} t^\alpha=\aa\alpha=\sum_{i=0}^d a_i\alpha_i.$$
For every non-zero polynomial $f=\sum_{i=1}^k \lambda_i t^{\alpha_i}$ we set $\deg_\aa(f)=\max\{ \aa \alpha_i : \lambda_i\neq 0\} $ and 
$\ini_\aa(f)=\sum_{ \aa \alpha_i=\deg_\aa(f)} \lambda_i t^{\alpha_i}$. Then for every ideal $I$ one defines the initial ideal $\ini_\aa(I)$ of $I$ with respect to $\aa$ to be: 

$$\ini_\aa(I)=(\ini_\aa(f) : f\in I, \ \ f\neq 0).$$

Similarly, given a term order $\tau$, we denote by $\ini_\tau(I)$ the ideal of the initial monomials of elements of $I$. Given $\aa\in \QQ_{\geq 0}^{d+1}$ the term order defined by 
$$t^\alpha\geq t^\beta \mbox { iff }\aa\alpha>\aa\beta \mbox{ or } (\aa\alpha=\aa\beta \mbox{ and } t^\alpha\geq t^\beta \mbox{ wrt } \tau)$$
is denoted by $\tau a$. 

One easily shows that $\ini_\tau(\ini_\aa(I))=\ini_{\tau\aa}(I)$. Hence $\ini_\aa(I)$ and $I$ have a common monomial initial ideal. This shows part (1) of the following lemma. 

\begin{lemma} Let $I$ be a homogeneous ideal with respect to the ordinary grading $\deg t_i=1$ and let $\aa\in \QQ_{\geq 0}^{d+1}$. Then 
\begin{itemize} 
\item[(1)] $S/I$ and $S/\ini_\aa(I)$ have the same Hilbert function. 
\item[(2)] $ \depth S/\ini_\aa(I)\leq \depth S/I$.
\end{itemize} 
\end{lemma}

Part (2) follows from the standard one-parameter flat family argument, see \cite[Chap.15]{E} or \cite{BC} for details.

\begin{definition} \label{P} 
Let $P$ be the ideal of the rational normal curve of $\PP^d$ in its standard embedding, i.e. $P$ is the
kernel of the $K$-algebra map $S=K[t_0, t_1,\dots , t_d]\to K[x,y]$ sending $t_i$ to $x^{d-i}y^i$. 
\end{definition} 
The ideal $P$ is minimally generated by the $2$-minors of the matrix 

$$T_d=\left( \begin{array}{cccccc}
t_0 & t_1 & t_2 & \dots &\dots & t_{d-1} \\
t_1 & t_2 & \dots &\dots & t_{d-1}& t_{d} 
\end{array} 
\right) $$
and it contains the binomials of the form 
$$t_{i_1}t_{i_2}\cdots t_{i_k}- t_{j_1}t_{j_2}\cdots t_{j_k}$$ 

with 

$$0\leq i_v, j_v \leq d \quad \mbox{ and } \quad i_1+i_2+\cdots+ i_k=j_1+j_2+\cdots+ j_k.$$

The Hilbert series of $S/P$ is: 

$$\frac{ 1+(d-1)z}{(1-z)^2}.$$

\section{Contracted ideals in dimension $2$} \label{contra2} 
 
We briefly recall from \cite[App.5]{ZS}, \cite[Chap.14]{HS} and \cite{CDJR} few facts about contracted ideals. As we deal only with homogeneous ideals, we will state the results in the graded setting. 

Assume $K$ is an algebraically closed field. Let $R=K[x,y]$ and denote by $\mm$ its maximal homogeneous ideal. 
An $\mm$-primary homogeneous ideal $I$ of $R$ is said to be contracted if it is contracted from a quadratic extension, that is, if there exists a linear form $z\in R$ such that $I=IS\cap R,$ where $S=R[\mm/z]$. The property of being contracted can be described in several ways, see \cite[Prop.3.3]{CDJR} for instance. To a contracted ideal $I$ one associates a form, the characteristic form of $I$, defined as $\GCD(I_d)$ where $d$ is the initial degree of $I$. For our goals, it is important to recall the following definition and theorem.

\begin{definition} 
\label{Smodu}
Let $I$ be a homogeneous $\mm$-primary ideal in $R$ and let $Q=IR_{\mm}$. Let $J\subset R_{\mm}$ be a minimal reduction of $Q$. The deviation of $I$ is the length of $Q^2/JQ$. 
It will be denoted it by $V(I)$. 
\end{definition}

\begin{theorem} 
\label{redunum}
Let $I$ be a homogeneous $\mm$-primary ideal in $R$. One has: 
\begin{itemize} 
\item[(1)] $\gr_I(R)$ is Cohen-Macaulay if and only if $V(I)=0$. 
\item[(2)] $V(I)=e_0(I)-\dim_K (R/I^2)+2\dim_K(R/I)$. 
\end{itemize} 
 \end{theorem}
 
 See \cite[Prop.2.6, Thm.A]{HM} for a proof of (1) and \cite[Lemma 1]{V} for a proof of (2). 
Similar results are proved also in \cite{Ver}. 

We recall now Zariski's factorization theorem for contracted ideals and a related statement, \cite[Cor.3.14]{CDJR}, concerning associated graded rings. In our setting they can be stated as follows. 
 
\begin{theorem} 
\label{ZaGr}
\begin{itemize}
\item[(1)] Any contracted ideal $I$ has a factorization $I=L_1\cdots L_s$ where the $L_i$ are homogeneous $\mm$-primary contracted ideals with characteristic form of type $\ell_i^{\alpha_i}$ for pairwise linearly independent linear forms $\ell_1,\dots,\ell_s$. 
\item[(2)] With respect to the factorization in (1) one has 
$$\depth \gr_I(R)=\min\{\depth \gr_{L_i}(R) : i=1,\dots,s\}.$$
\end{itemize} 
 \end{theorem}

\begin{lemma}
\label{Fshort} 
The fiber cone $F(I)$ of a contracted ideal $I$ has a short $h$-vector. Its Hilbert series is
$(1+(d-1)z)/(1-z)^2$, where $d$ is the initial degree of $I$. 
\end{lemma} 
\begin{proof} A contracted ideal of initial degree $d$ is minimally generated by $d+1$ elements and products of contracted ideals are contracted. The initial degree of $I^k$ is $kd$. Hence $I^k$ has $kd+1$ minimal generators. It follows that the Hilbert series of $F(I)$ is $(1+(d-1)z)/(1-z)^2$. 
\end{proof}

A monomial $\mm$-primary ideal $I$ of $R=K[x,y]$ can be encoded in various ways. We use the following. Let $d\in \NN$ be such that $x^d\in I$ and, for $i=0,\dots, d,$ set $a_i(I)=\min\{ j : x^{d-i}y^j\in I \}$. Then we have $0=a_0(I)\leq a_1(I) \leq \dots \leq a_d(I)$. Obviously, the map 
$$I\to \aa=(a_0(I),\dots,a_d(I))$$ 
establishes a bijective correspondence between the set of $\mm$-primary monomial ideals containing $x^d$ and the set of weakly increasing sequences $\aa=(a_0,\dots,a_d)$ of non-negative integers with $a_0=0$. The inverse map is 
$$\aa=(0=a_0\leq a_1\dots \leq a_d) \to (x^{d-i}y^{a_i} : i=0,\dots, d).$$
It is easy to see that if $\aa$ corresponds to $I$, then 
$\dim_K R/I=\sum_{i=0}^d a_i$. 
Furthermore, if $\aa'$ corresponds to $J$, then the sequence associated to the product $IJ$ is $(c_0,c_1,c_2,\dots)$ where 
$c_i=\min\{ a_j+a'_k : j+k=i\}$. In particular we have the following lemma.

\begin{lemma}\label{formH}
Let $I$ be a monomial ideal and $\aa=(a_0,\dots,a_d)$ be the corresponding sequence. Then the Hilbert function of $I$ is given by: 
$$H^1(I,k)=\sum_{i=0}^{(k+1)d} \min\{ a_{j_1}+\dots +a_{j_{k+1}} : j_1+\dots+j_{k+1}=i\}.$$
\end{lemma} 
\bigskip
 
 We set 
$$b_i(I)=a_i(I)-a_{i-1}(I)$$
 for $i=1,\dots,d$ and observe that the ideal $I$ can be as well described via the sequence $b_1(I),\dots, b_d(I)$ of non-negative integers. 

A monomial ideal $I$ is a lex-segment ideal if $x^iy^j\in I$ for some $j>0$ implies $x^{i+1}y^{j-1}\in I$. The $\mm$-primary lex-segment ideals are contracted and correspond exactly to strictly increasing $a$-sequences (equivalently, positive $b$-sequences) in the above correspondence, provided one takes $d=\min\{ j\in \NN : x^j\in I\}$.

\begin{remark} 
\label{sonolex}
With respect to suitable coordinate systems the ideals $L_i$ appearing in Theorem \ref{ZaGr} are lex-segment ideals. 
\end{remark} 

It follows from Theorem \ref{ZaGr} and Remark \ref{sonolex} that the study of the depth of the associated graded ring of contracted ideals boils down to the study of the depth of $\gr_{L}(R)$ for a lex-segment ideal $L$. 
Since $R$ is regular, one has $\depth \gr_{L}(R)=\depth \Rees(L)-1$, see \cite[Cor.2.7]{HM}. Therefore we can as well study the depth of $ \Rees(L)$ for lex-segment ideals $L$. 
 Trung and Hoa gave in \cite{TH} a characterization of the Cohen-Macaulay property of affine semigroup rings. Since $\Rees(L)$ is an affine semigroup ring one could hope to use their results to describe the lex-segment ideals $L$ such that $\Rees(L)$ is Cohen-Macaulay. In practice, however, we have not been able to follow this idea. 
 
Let $L$ be the lex-segment ideal with associated $a$-sequence $\aa=(a_0,\dots,a_d)$ and $b$-sequence $(b_1,\dots,b_d)$. We present $\Rees(L)$ as a quotient of $R[t_0,\dots, t_d]$ by the $R$-algebra map 
$$\psi: R[t_0,\dots, t_d]\to \Rees(L)\subset R[t]$$
obtained by sending $t_i\to x^{d-i}y^{a_i}t$. 
Set ${\bf 1}=(1,1,\dots,1)\in \NN^{d+1}$ and ${\bf d}=(d,d-1,d-2,\dots,0)\in \NN^{d+1}$. 

\begin{lemma} \label{genera}
 With the above notation, $\ker \psi$ is generated by the following binomials: 
\begin{itemize}
\item[(1)] $xt_i-y^{b_i}t_{i-1}$ with $i=1,\dots, d.$ 
\item[(2)] $t^{\alpha}-y^u t^{\beta}$ where $\alpha,\beta \in \NN^{d+1}$ satisfy
$$\left\{
\begin{array}{l}
{\bf 1}(\alpha-\beta)=0\\
{\bf d}(\alpha-\beta)=0\\ 
u=\aa(\alpha-\beta) \geq 0. 
\end{array}
\right.
$$
\end{itemize}
\end{lemma} 

\begin{proof} Let $J$ be the ideal generated by the binomials of type (1) and (2). Obviously $J\subseteq \ker\psi$. Since $\ker \psi$ is generated by the binomials it contains, it is enough to show that every binomial $M_1-M_2\in \ker \psi$ with $\GCD(M_1,M_2)=1$ belongs to $J$. Up to multiples of elements of type (1), we may assume that if $x$ divides one of the $M_i$, say $M_1$, then $M_1=x^iy^jt_0^ k$. But this clearly contradicts the fact that $M_1-M_2\in \ker \psi$. In other words, every binomial in $\ker\psi$ is, up to multiples of elements of type (1), a multiple of an element of type (2). 
\end{proof}

\begin{lemma} Let $L$ be a lex-segment ideal and $\aa=(a_0,\dots,a_d)$ its associated $a$-sequence, then we have: 
\begin{itemize}
\item[(1)] $\Rees(L)/(y)\Rees(L) =K[x,t_0,\dots,t_d]/x(t_1,\dots,t_d)+\ini_\aa(P)$,
\item[(2)] $F(L)=K[t_0,\dots,t_d]/\ini_\aa(P)$,
\item[(3)] $\depth \gr_L(R)=\depth \Rees(L)-1=\depth F(L)$, 
\end{itemize} 
where $P$ is the ideal introduced in Definition \ref{P}. \end{lemma} 

\begin{proof} Set $F=F(L)$, $G=\gr_L(R)$ and ${\mathcal R}=\Rees(L)$. First note that (2) follows from (1) since $F={\mathcal R}/(x,y){\mathcal R}$. To prove (1) we have to show that 

$$\ker\psi+(y)=x(t_1,\dots,t_d)+\ini_\aa(P)+(y).$$

For the inclusion $\subseteq$ we show that the generators of $\ker\psi$ of type (1) and (2) in Lemma \ref{genera} belong to the ideal on the right hand side. This is obvious for those of type (1). For those of type (2), note that for any such $t^\alpha-y^u t^\beta$ one has $t^\alpha- t^\beta\in P$ and $\ini_\aa(t^\alpha- t^\beta)=t^\alpha$ if $u>0$ and $\ini_\aa(t^\alpha- t^\beta)=t^\alpha- t^\beta$ if $u=0$. 

The inclusion $\supseteq$ for the elements of $x(t_1,\dots,t_d)$ is obvious. Further, since $P$ is generated by binomials, one knows that $\ini_\aa(P)$ is generated by $\ini_\aa(t^\alpha-t^\beta)$ with $t^\alpha-t^\beta\in P$, see \cite[Chap.1]{St}. If $\aa(\alpha-\beta)=0$, then $\ini_\aa(t^\alpha-t^\beta)=t^\alpha-t^\beta$ and $t^\alpha-t^\beta\in \ker\psi$. If instead $u=\aa(\alpha-\beta)>0$, then $\ini_\aa(t^\alpha-t^\beta)=t^\alpha$ and $t^\alpha-y^ut^\beta\in \ker\psi$ and so $t^\alpha\in \ker\psi+(y)$. 

To prove (3) note that $P\subseteq (t_1,\dots, t_d)$ and hence $\ini_\aa(P)\subseteq (t_1,\dots, t_d)$. It follows that 

$$(t_1,\dots, t_d)\subseteq (x(t_1,\dots,t_d)+\ini_\aa(P)):x\subseteq 
(t_1,\dots,t_d):x=(t_1,\dots,t_d).$$
Hence 
$$(t_1,\dots, t_d)=(x(t_1,\dots,t_d)+\ini_\aa(P)):x$$
and we get a short exact sequence 
$$0\to K[x,t_0](-1)\to {\mathcal R}/(y){\mathcal R}\to F\to 0.$$

By Corollary \ref{Fshort} the ring $F$ is $2$-dimensional with short h-vector. It follows that the same is true for ${\mathcal R}/(y){\mathcal R}$ (with respect to the standard grading). Using the depth formula for short exact sequences \cite[Prop.1.2.9]{BH}, we have that if $\depth {\mathcal R}/(y){\mathcal R}$ is $0$ or $1$, then $\depth F=\depth {\mathcal R}/(y){\mathcal R}$. Finally, if ${\mathcal R}/(y){\mathcal R}$ is Cohen-Macaulay then $\reg( {\mathcal R}/(y){\mathcal R})=1$ and it follows that $\reg F=1$. Then from Lemma \ref{CMvsCM} we can conclude that $F$ is Cohen-Macaulay. 

We have shown that $\depth {\mathcal R}-1=\depth {\mathcal R}/(y){\mathcal R}=\depth F$. Since by \cite[Cor.2.7]{HM} $\depth G=\depth {\mathcal R}-1$, the proof of (3) is complete. 
\end{proof} 

Summing up, we have shown the following proposition. 

\begin{proposition} \label{veryimpo}
Let $L$ be a lex-segment ideal in $R=K[x,y]$ with associated $a$-sequence $\aa=(a_0,\dots,a_d)$. Then 
$$\depth \gr_L(R)=\depth K[t_0,\dots,t_d]/\ini_\aa(P)$$
where $P$ is the ideal of Definition \ref{P}. 
\end{proposition} 

\section{Cohen-Macaulay initial ideals of the ideal of the rational normal curve}\label{CM for P}

The results of the previous sections show that the study of the contracted ideals of $K[x,y]$ whose associated graded ring is Cohen-Macaulay is equivalent to the study of the initial ideals of $P$ defining Cohen-Macaulay rings. 
In this section we describe the initial ideals of $P$ (with respect to vectors and in the given coordinates) defining Cohen-Macaulay rings. 
 In the following we will say that an ideal $I$ is Cohen-Macaulay if the quotient ring defined by $I$ is Cohen-Macaulay. 
 The steps of the classification are:

\begin{itemize}
\item[(1)] Classify the $2$-dimensional Cohen-Macaulay monomial ideals with minimal multiplicity. 
\item[(2)] Among the ideals of (1) identify those of the form $\ini_\tau(P)$ for some term order $\tau$.
\item[(3)] Identify the vectors $\aa\in \QQ_{\geq 0}^{d+1}$ such that $\ini_\aa(P)$ is Cohen-Macaulay (monomial or not).
\end{itemize} 

 We start by classifying the $2$-dimensional Cohen-Macaulay square-free monomial ideals with minimal multiplicity. 
Square-free monomial ideals are in bijective correspondence with simplicial complexes. In particular,
square-free monomial ideals defining algebras of Krull dimension $2$ are in bijective correspondence
with simplicial complexes of dimension $1$, that is, graphs. The correspondence goes like this: to any
graph with vertex set $V$ and edge set $E$ one associates the monomial ideal on variables $V$ whose
generators are the products $xy$ such that $\{x,y\}$ is not in $E$ and by the square-free monomials
of degree $3$. 

Recall that a graph $G$ with $n$ vertices and $e$ edges is a tree if it satisfies the following equivalent conditions: 

\begin{itemize} 
\item[(1)] For every $x$ and $y$ distinct vertices there exists exactly one path in $G$ connecting
$x$ to $y$. 
\item[(2)] $G$ is connected and $n-e=1$. 
\item[(3)] $G$ is connected and if we remove any edge the resulting graph is disconnected. 
\end{itemize}

\begin{lemma}\label{trees} The $2$-dimensional Cohen-Macaulay square-free monomial ideals with minimal multiplicity correspond to trees. 
\end{lemma} 

\begin{proof} Let $G$ be a graph with $n$ vertices and $e$ edges, let $A$ be the corresponding quotient ring. The Hilbert series of $A$ is given by 
$$1+\frac{nz}{(1-z)}+\frac{ez^2}{(1-z)^2}$$
see \cite[5.1.7]{BH}. This implies immediately that $A$ has minimal multiplicity if and only if $n-e=1$. The Cohen-Macaulay property of $A$ corresponds to connectedness of $G$, see \cite[5.1.26]{BH}. So we are looking at connected graphs with $n-e=1$, that is, trees. \end{proof} 

Next we extend our characterization from the square-free monomial ideals to general monomial ideals. We will make use of the following lemma whose easy proof belongs to the folklore of the subject. 

\begin{lemma}\label{sizi} Let $I$ be a monomial ideal generated in degree $2$ and with linear
syzygies. Let $x,y,z$ be variables. We have: 
\begin{itemize}
\item[(1)] If $x^2,y^2\in I$ then $xy\in I$.
\item[(2)] If $x^2, yz\in I$ then either $xy\in I$ or $xz\in I$. 
\end{itemize} 
\end{lemma} 

\begin{proof} Say $I$ is generated by monomials $m_1=x^2$ , $m_2=y^2$ and other monomials $m_3,..., m_s$. 
Take a free module $F$ with basis $e_1,e_2,\dots e_s$ and map $e_j$ to $m_j$. The syzygy module
$\Syz(I)$ is generated by the reduced Koszul relations $ae_i-be_j$ with $a=m_j/\GCD(m_i,m_j)$ and
$b=m_i/\GCD(m_i,m_j)$. By assumption we know that $\Syz(I)$ is generated by the elements $ae_i-be_j$
with $\deg(a)=1$, call this set $G$. Now $y^2e_1-x^2e_2$ is in $\Syz(I)$ and so can be written as $\sum
v_{ij} (ae_i-be_j)$ where the sum is extended to the elements $ae_i-be_j\in G$. It follows that 
there must be in $G$ an element of the form $ye_1-be_i$. We deduce that $m_i/\GCD(m_i,x^2)=y$, forcing
$m_i$ to be $xy$. Similarly one proves (2). 
\end{proof} 

The crucial inductive step is encoded in the following lemma. 

\begin{lemma}
\label{induct} Let $I\subset S=K[x_1,\dots, x_n]$ be a monomial ideal such that $S/I$ is Cohen-Macaulay of dimension $2$ with minimal multiplicity. Let 
$x_i$ be a variable such that $x_i^2\in I$. Set $S'=K[x_j : 1\leq j\leq n, j\neq i]$. Then: 
\begin{itemize}
\item[(1)] $I:(x_i)$ is generated by exactly $n-2$ variables. 
\item[(2)] Write $I+(x_i)=J+(x_i)$ where $J$ is a monomial ideal of $S'$. 
 Then $S'/J$ is a $2$-dimensional Cohen-Macaulay ring with minimal multiplicity. 
\end{itemize} 
\end{lemma} 

\begin{proof} First we show that $I:(x_i)$ is generated by variables. Let $m$ be one of the generators
of $I$. We have to show that if $x_i$ does not divide
$m$ then there exists a variable $x_j$ such that $x_j|m$ and $x_ix_j\in I$.
Let $V$ be the set of the variables whose square is not in $I$ and let $Q$ be the
set of the variables whose square is in $I$. If $m$ is divisible by a variable in $Q$, then we are done by
Lemma \ref{sizi}(1). Otherwise, if $m$ is not divisible by a variable in $Q$, then $m=x_jx_k$ with
$x_j,x_k\in V$, $j\not=k$. By Lemma \ref{sizi}(2) we have that either $x_ix_j$ or $x_ix_k$ is in $I$. Knowing that
$I:(x_i)$ is generated by variables we deduce that $I:(x)$ is a prime ideal, so an associated prime of the Cohen-Macaulay ideal $I$. Hence the codimension of $I:(x_i)$ is $n-2$. This proves
(1). The standard short exact sequence:
$$0\to S/I:(x_i)(-1)\to S/I\to S'/J \to 0$$
shows that the Hilbert series of $S'/J$ is $1+(n-3)z/(1-z)^2$. Hence $S'/J$ has a short $h$-vector. 
Furthermore, the exact sequence implies that the regularity of $S'/J$ is $1$, that is, 
$\reg(J)=2$. By Lemma \ref{CMvsCM} we conclude that
$S'/J$ is Cohen-Macaulay. 
\end{proof}

\begin{corollary}\label{caroca} With the assumption of Lemma \ref{induct} and the notation of its proof, there exist $x_j, x_k\in V$ such that $I:(x_i)=( x_v : 1\leq v\leq n \mbox{ and } v\not\in
\{j,k\})$ and $x_jx_k\not\in I$. 
\end{corollary} 

\begin{proof} We know by Lemma \ref{induct} that there are variables $x_j$ and $x_k$ so that 
$I:(x_i)=( x_v : 1\leq v\leq n \mbox{ and } v\not\in \{j,k\})$. If $x_jx_k\in I$, then $x_jx_k\in
I:(x_i)$. This is a contradiction. 
\end{proof} 

\begin{definition} Let $V$ and $Q$ be disjoint sets of variables. Let $G$ be a tree with vertex set $V$
and edge set $E$. Let $\phi:Q \to E$ be a map. Let $J$ be the square-free monomial ideal associated
with $G$. Let $H=(Q)^2+( xy : x\in Q , y\in V \mbox{ and } y\not\in \phi(x) )$. We define
$$I(G,\phi)=J+H.$$
\end{definition}

\begin{example} Let $V=\{v_1,v_2,v_3,v_4\}$ and $Q=\{q_1,q_2,q_3,q_4,q_5\}$. Let $G$ be the tree on
$V$ with edges $E=\{e_1, e_2, e_3\}$ with $e_1=\{v_1,v_2\}, e_2=\{v_1, v_3\}, e_3=\{v_3,v_4\}$
and let $\phi:Q\to E$ be the map sending $q_1,q_2,q_3,q_4,q_5$ to $e_2,e_1,e_1,e_3,e_2$, 
respectively. Then $J=(v_1v_4, v_2v_3, v_2v_4)$ and
$H=(q_1,q_2,q_3,q_4,q_5)^2+q_1(v_2,v_4)+q_2(v_3,v_4)+q_3(v_3,v_4)+q_4(v_1,v_2)+q_5(v_2,v_4)$.
\end{example} 

In the next proposition we achieve the first step of the classification. 

\begin{proposition} \label{class1}
Let $I\subset S$ be a monomial ideal. The following conditions are equivalent:
\begin{itemize} 
\item[(1)] there exist a tree $G$ and a map $\phi:Q\to E$ such that $I=I(G,\phi)$. 
\item[(2)] $S/I$ is a $2$-dimensional Cohen-Macaulay ring with minimal multiplicity. 
 \end{itemize} 
\end{proposition}

\begin{proof} First we show that every ideal of type $I(G,\phi)$ defines a Cohen-Macaulay, $2$-dimensional
ring with minimal multiplicity. We proceed by induction on the cardinality of $Q$. If $Q$ is empty, then the ideal
$I$ is $2$-dimensional Cohen-Macaulay with minimal multiplicity by Lemma \ref{trees}. Now assume
$Q$ is not empty and pick $q\in Q$. By construction $I(G,\phi):(q)$ is generated by $Q\cup V\setminus
\phi(q)$ and $I(G,\phi)+(q)=I(G,\phi')+(q)$ where $\phi'$ is the restriction of $\phi$ to
$Q'=Q\setminus \{q\}$. By induction we know that $I(G,\phi')\subset K[V,Q']$ is $2$-dimensional Cohen-Macaulay with minimal multiplicity. The short exact sequence 
$$0\to K[V,Q]/I(G,\phi):(q)(-1)\to K[V,Q]/I(G,\phi) \to K[V,Q']/I(G,\phi')\to 0$$
shows that $I(G,\phi)$ is $2$-dimensional with minimal multiplicity and $K[V,Q]/I(G,\phi)$ is Cohen-Macaulay since
both $K[V,Q]/I(G,\phi):(q)$ and $K[V,Q']/I(G,\phi')$ are $2$-dimensional and Cohen-Macaulay. 

We show now that every $2$-dimensional Cohen-Macaulay monomial ideal $I$ with minimal multiplicity is of the
form $I(G,\phi)$. Let $Q$ be the set of variables whose square is in $I$ and $V$ the remaining
variables. We argue by induction on the cardinality of $Q$. If $Q$ is empty, then $I$ is square-free
and we know that $I$ is associated to a tree. If $Q$ is not empty, let $q\in Q$. Write $I+(q)=J+(q)$
with $J$ a monomial ideal not involving the variable
$q$. By Lemma \ref{induct} we know that $J\subset K[V,Q\setminus\{q\}]$ is $2$-dimensional Cohen-Macaulay with minimal
multiplicity. Therefore there exists a tree $G$ with vertices $V$ and edges $E$, and a map $\phi':
Q\setminus\{q\}\to E$ so that $J=I(G,\phi')$. On the other hand $I=J+q(I:q)$ and by Corollary \ref{caroca} there 
are $x,y\in V$ so that $I:(q)=(Q\cup V\setminus \{x,y\})$ with $xy\not\in I$. Hence $\{x,y\}$ belongs
to $E$ and we extend $\phi'$ to $Q$ by sending $q$ to $\{x,y\}$. Call $\phi$ the resulting map 
$\phi:Q\to E$. By construction, $I(G,\phi)=I$. 
\end{proof}

Now we come to the second step of our classification: describe the Cohen-Macaulay monomial initial ideals of $P$. Let $I$ be a monomial Cohen-Macaulay initial ideal of $P$. Since $P$ is $2$-dimensional with a short $h$-vector then also $I$ is so. By what we have seen above, there exist $G$ and $\phi$ so that $I=I(G,\phi)$. We want to describe which pairs $ (G, \phi) $ arise in this way. 

\begin{lemma}
\label{regole} 
 Let $I$ be a monomial initial ideal of $P$ generated in degree $2$. 
\begin{itemize} 
\item[(1)] For every $v$, $0\leq v\leq 2d,$ there exists exactly one monomial $t_it_j$ such that
$i+j=v$ and $t_it_j\not\in I$. In particular, $t_0^2$ and $t_d^2$ are not in $I$. 
\item[(2)] Let $0\leq i<j<k\leq d$. Assume that $t_i^2, t_j^2$ and $t_k^2$ are not in $I$. Then
$t_it_k\in I$. 
\item[(3)] Let $0\leq i<j<k\leq d$. Assume that $t_i^2\not\in I$ and $t_st_j\not\in I$ for every $s$ with $i\leq s\leq j$. Then $t_it_k\in I$. 
\end{itemize} 
\end{lemma} 
\begin{proof} (1) By definition, $K[t_0,\dots,t_d]/P$ is bigraded by setting $\deg(t_i)=(d-i,i)$ and
$1$-dimensional in each bi-homogeneous component. Therefore also the ring $K[t_0,\dots,t_d]/I$ is bigraded 
 and $1$-dimensional in each bi-homogeneous component. This proves the assertion. 

(2) Since
$t_i^{k-j}t_k^{j-i}-t_j^{k-i}\in P$ we have that either 
$t_i^{k-j}t_k^{j-i}$ or $t_j^{k-i}$ belongs to $I$. If $t_j^{k-i}$ belongs to $I$, since $I$ is
generated in degree $2$, then
$t_j^2$ belongs to $I$, a contradiction. So we have $t_i^{k-j}t_k^{j-i}\in I$. Since $t_i^2$ and
$t_k^2$ are not in $I$ we may conclude that $t_it_k\in I$. 

(3) We show that there exist a positive integer $a$ and $i\leq s\leq j$ so that 
$t_i^at_k-t_st_j^a\in P$. Then since, by assumption, the factors of degree $2$ of $t_st_j^a$ are 
not in $I$, we have that a factor of degree $2$ of $t_i^at_k$ is in $I$. Since $t_i^2\not\in I$, it follows that $t_it_k\in
I$. To show that $a$ and $s$ as above exist, note that the condition $t_i^at_k-t_st_j^a\in P$
translates into $s=k-a(j-i)$. So we have to show that there exists a positive integer $a$ such that 
$$i\leq k-a(j-i) \leq j,$$
equivalently 
$$\frac{k-j}{j-i} \leq a \leq \frac{k-i}{j-i}.$$
But this interval has length $1$ and so it contains an integer which is positive since $i<j<k$. 
\end{proof} 

The matrix $T_d$ appearing in the following theorem is defined in Section \ref{notpre}. 
 
\begin{theorem}\label{moninidCM} Let $I$ be a monomial ideal of $K[t_0,\dots, t_d]$. 
The following conditions are equivalent: 
\begin{itemize} 
\item[(1)] There exists a sequence $0=i_0<i_1<i_2\dots<i_k=d$ such that $I(G,\phi)=I$ 
where $V=\{t_{i_0},
t_{i_1}, \dots, t_{i_k}\}$, $Q=\{t_0,\dots,t_d\}\setminus V$, $G$ is the tree (a line) 
with vertex set $V$ and with edge set $E=\{ \{t_{i_j}, t_{i_{j+1}}\} : 
0\leq j\leq k-1\}$ and the map $\phi:Q\to E$ sends $t_s$ to $\{t_{i_j}, t_{i_{j+1}}\}$
where $i_j<s<i_{j+1}$. 
\item[(2)] There exists a sequence $0=i_0<i_1<i_2\dots<i_k=d$ such that $I$ is generated by 
 \begin{itemize}
\item[(i)] the main diagonals $t_{v-1}t_r$ of the $2$-minors of the matrix $T_d$ with column indices
$v,r$ such that $v\leq i_j< r$ for some $j$,
\item[(ii)] the anti-diagonals $t_{v}t_{r-1}$ of the $2$-minors of the matrix $T_d$ with column indices
$v,r$ such that $i_j<v<r\leq i_{j+1}$ for some $j$. 
 \end{itemize} 
\item[(3)] $I$ is a Cohen-Macaulay initial ideal of $P$. 
\end{itemize} 
\end{theorem} 

\begin{proof} 
To prove that (1) and (2) are equivalent is just a direct check. 
To prove that (1) and (2) imply (3), it is enough to describe a term order
$\tau$ such that $\ini_\tau(P)=I(G,\phi)$. Indeed, the inclusion $\ini_\tau(P)\supseteq I(G,\phi)$ is
enough because the two ideals have the same Hilbert function. To do so consider a vector $b=(b_1,b_2,\dots, b_d)$ such that $b_r>b_v$ if $v\leq i_j< r$ for some $j$ and $b_v>b_r$ if $i_j<v<r\leq i_{j+1}$ for some $j$. In Remark \ref{next} we will show a canonical way to construct such a vector. Define $\aa=(a_0,a_1,\dots,a_d)$ by setting $a_0=0$ and $a_i=\sum_{j=1}^i
b_j$. Consider a term order $\tau$ refining (no matter how) the order defined by the vector $\aa$. We claim that the initial term with respect to $\tau$ of the minor with column indices $v,r$
($v<r$) of the matrix $T_d$ is the one prescribed by (2). The minor with those column indices is
$t_{v-1}t_r-t_vt_{r-1}$. The weights with respect to $\aa$ of the two terms are $a_{v-1}+a_r$
and $a_{v}+a_{r-1}$. Hence $a_{v-1}+a_r>a_{v}+a_{r-1}$ iff $b_r>b_v$ and $a_{v-1}+a_r<a_{v}+a_{r-1}$
 iff $b_r<b_v$. Therefore, by construction, the initial forms of the minors of $T_d$ with respect to $\aa$ are exactly the monomials prescribed by (2). 

We prove now that (3) implies (1). Assume that $I$ is a Cohen-Macaulay initial ideal of $P$. Then $I$ has
regularity $2$, in particular is generated by elements of degree $2$. Set
$V=\{ t_i : t_i^2 \not\in I\}$ and $Q=\{ t_i : t_i^2 \in I\}$. We know by Lemma \ref{regole}(1) that $t_0$
and $t_d$ are in $V$. So we may write $V=\{t_{i_0},t_{i_1},\dots, t_{i_k}\}$
with $0=i_0<i_1<i_2\dots<i_k=d$. Since $I$ is $2$-dimensional Cohen-Macaulay with minimal multiplicity, it has
the form $I=I(G', \phi')$. We prove that
$G'=G$ and $\phi'=\phi$ where $G$ and $\phi$ are those described in (1). Note that, by 
Lemma \ref{regole}(2), we have that $t_{i_v}t_{i_r}\in I$ whenever $r-v>1$ so that $\{t_v, t_r\}$ is not an
edge of $G'$. This implies that the underlying tree $G'$ is exactly the line $G$. It remains to prove
that $\phi=\phi'$. In other words, we have to prove that for every $t_j\in Q$, say $i_r<j<i_{r+1}$,
one has $t_jt_{i_r}\not\in I$ and $t_jt_{i_{r+1}}\not\in I$. By contradiction, let $j$ be the
smallest element of $Q$ that does not satisfy the required condition, say $i_r<j<i_{r+1}$. 
Now we claim the following: \medskip 

\noindent { \bf Claim} If $0 \leq v<r$ and $q\geq j$ then $t_{i_v}t_q\in I$. 
\medskip 

To prove the claim note that, by the choice of $j$, we know that $t_{s}t_{i_{v+1}}\not\in I$ 
for every $i_v\leq s \leq i_{v+1}$ and $i_{v+1}\leq i_r<j$. Therefore we may apply Lemma \ref{regole}(3) to
the indices $i_v, i_{v+1}, q$ and we conclude that $t_{i_v}t_q\in I$. 

In particular, for $q=j$ the claim says that $t_{i_v}t_j\in I$ for every $v<r$. So $\phi'(t_j)$ must be
an edge of the form $\{t_{i_u}, t_{i_{u+1}} \}$ with $u>r$. It follows that $t_{i_r}t_j\in I$. 
But, according to Lemma \ref{regole}(1), there exists (exactly one, but we do not need this) a 
monomial $t_at_b$ (say $a\leq b$) such that $a+b=i_r+j$ and $t_at_b\not\in I$. We distinguish now
three cases: 

\medskip 

\noindent{ \bf Case} 1) 
 $i_r<a$ and $b<j$, so that $a$ and $b$ are both in $Q$. A contradiction (the square of $Q$ is
contained in $I$). 

\medskip 

\noindent{ \bf Case} 2) $a<i_r$, $b>j$ and $a\in V$. The claim above says
that $t_at_b\in I$. A contradiction.

\medskip 

\noindent{ \bf Case} 3) $a<i_r$, $b>j$ and $a\in Q$ . Say $i_u<a<i_{u+1}$. By induction we know
that $\{ t_v : t_vt_a\not\in I\}=\{ t_{i_u}, t_{i_{u+1}}\}$ and so $b>j>i_r\geq i_{u+1}$. Therefore 
$t_at_b\in I$. A contradiction. 

This concludes the proof. 
\end{proof}

\begin{remark}\label{next}
Given $0=i_0<i_1<i_2\dots<i_k=d$, in the proof of Theorem \ref{moninidCM} we have used a vector $b=(b_1,b_2,\dots,b_d)\in \QQ_{\geq 0}^d$ with the property that $b_r>b_v$ if $v\leq i_j< r$ for some $j$ and $b_v>b_r$ if $i_j<v<r\leq i_{j+1}$ for some $j$. Of course there are many vectors with that property. But there is just one ``permutation" vector (i.e. its entries are the number $1,\dots,d$ permuted in some way) with that property. 
It arises as following: for
$j=1,\dots, k$ consider the vector 
$c_j=(i_j,i_j-1,\dots, i_{j-1}+1)$ and define the vector $b$ to be the
concatenation of $c_1,c_2\dots,c_k$. 
\end{remark}

The following theorem summarizes what we have proved so far. 

\begin{theorem} \label{riassunto} 
\begin{itemize} 
\item[(1)] The ideal $P$ has exactly $2^{d-1}$ distinct Cohen-Macaulay monomial initial ideals.
\item[(2 )] They are in bijective correspondence with the sequences $0=i_0<i_1<\dots<i_k=d$, i.e. 
with their radical. 
\item[(3)] Each of them can be obtained with a term order associated to a vector
$\aa=(a_0,a_1,\dots,a_d)\in \NN^{d+1}$ with $0=a_0<a_1<\dots<a_d$. 
\item[(4)] Each of them is obtained by taking the appropriate initial terms of the $2$-minors of
the matrix $T_d$. 
\item[(5)] The reduced
Gr\"obner basis of $P$ giving the Cohen-Macaulay initial ideal corresponding to the sequence 
$0=i_0<i_1<i_2\dots<i_k=d$ is the set of polynomials 
$$\begin{array} {ll} 
\underline{t_st_r}-t_{i_v}t_{s+r-i_v} & \mbox{ if } \quad 2i_v\leq s+r\leq i_v+i_{v+1}\\
\underline{t_st_r}-t_{i_{v+1}}t_{s+r-i_{v+1}} & \mbox{ if } \quad i_v+i_{v+1}\leq s+r\leq 2i_{v+1}
\end{array} 
$$
where the initial terms are underlined.
\end{itemize} 
\end{theorem} 

 \begin{definition}\label{thecone}
For every sequence $\ii=(i_0,i_1,\dots,i_k)$ with $0=i_0<i_1<\dots<i_k=d$ we denote by 
$C(\ii)$ the open cone in $\QQ_{\geq 0}^{d+1}$ of the points $(a_0,\dots,a_d)$ satisfying the inequalities 
$$\begin{array} {ll} 
a_s+a_r > a_{i_v}+a_{s+r-i_v} & \mbox{ if } 2 i_v\leq s+r \leq i_v+i_{v+1},\\
a_s+a_r >a_{i_{v+1}}+a_{s+r-i_{v+1}} & \mbox{ if } i_v+i_{v+1}\leq s+r \leq 2i_{v+1},
\end{array} 
$$
and by $\ol{C(\ii)}$ the corresponding closed cone, that is, the subset of $\QQ_{\geq 0}^{d+1}$ described by the inequalities above where $>$ is replaced throughout by $\geq$. 
\end{definition} 

 We are ready to state and prove the main theorem of this section. 

\begin{theorem}\label{nonmon} 
Let $\aa\in \QQ_{\geq 0}^{d+1}$. Then $\ini_\aa(P)$ is Cohen-Macaulay if and only if $\ini_\aa(P)$ has a Cohen-Macaulay initial monomial ideal. In other words, 
$$\{ \aa\in \QQ_{\geq 0}^{d+1} : \ini_\aa(P) \mbox{ is Cohen-Macaulay}\}=\cup_\ii \ \ol{C(\ii)}$$
where the union is extended to the set of the $2^{d-1}$ sequences $\ii=(0=i_0<i_1<\dots<i_k=d)$.
\end{theorem} 

\begin{proof} First we prove the inclusion $\supseteq$. Let $\aa \in \ol{C(\ii)}$. To see that 
$\ini_\aa(P)$ is Cohen-Macaulay it is enough to prove that it has a Cohen-Macaulay initial ideal. Just take $\aa'$ in the open cone ${C(\ii)}$ and check that $\ini_{\aa'}(\ini_\aa(P))=\ini_{\aa'}(P)$. This is easy since it is enough to check that $\ini_{\aa'}(\ini_\aa(P))\supseteq \ini_{\aa'}(P)$. 

In order to prove the opposite inclusion, let $\aa\in \QQ_{\geq 0}^{d+1}$ be such that $\ini_\aa(P)$ is Cohen-Macaulay. We have to show that $\aa\in \ol{C(\ii)}$ for some $\ii$. We know already that if $\ini_\aa(P)$ is monomial then $\aa\in C(\ii)$ where $\ii$ is the sequence of the indices $i$ such that no power of $t_i$ belongs to $\ini_\aa(P)$. So we are left with the case $\ini_\aa(P)$ is non-monomial. To treat this case we first note that, without loss of generality, we may assume that $\aa=(a_0,\dots,a_d)$ with $a_i\in \NN$, $a_0=0$, and $a_i<a_{i+1}$. This is because cleaning denominators and adding to $\aa$ multiples of the vectors $(1,\dots, 1)$ and $(0,1,2,\dots,d)$ do not change $\ini_\aa(P)$ nor the membership in the cones. 
Then we may associate to $\aa$ a lex-segment ideal $L$ in $R=K[x,y]$ as described in Section \ref{contra2}. We compute the deviation $V(L)$ of $L$ in terms of the $a_i$'s. It is well-know, see for instance \cite{DTVVW}, that the multiplicity $e_0(L)$ of $L$ is twice the area of the region $\RR_+^2\setminus \hull(L)$ where $\hull(L)$ is the Newton polytope of $L$, i.e. the convex hull of the set of elements $(a,b)\in \NN^2$ such that $x^ay^b\in L$. To determine $e_0(L)$ we describe the vertices of $\hull(L)$. The generators of $L$ are the elements $x^{d-i}y^{a_i}$. Set $i_0=0$ and assume that $i_t<d$ is already defined. Then we set 
$$m=\min\{ (a_j-a_{i_t})/(j-i_t) : j=i_t+1,\dots,d\}$$
 and 
$$i_{t+1}=\max\{j : i_t+1<j\leq d \mbox{ and } (a_j-a_{i_t})/(j-i_t)=m \}.$$
The procedure stops when we have reached, say after $k$ steps, $i_k=d$. By construction 
the points $(d-c,a_{c})$ with $c\in \{i_0,i_1,\dots, i_k\}$ are the vertices of $\hull(L)$. Taking into account that twice the area of the triangle with vertices $(0,0), (d-i_t,a_{i_t}), (d-i_{t+1}, a_{i_{t+1}})$ is 
$$a_{i_{t+1}}(d-i_t)-a_{i_{t}}(d-i_{t+1})$$
and that $a_0=0$, $i_0=0$ we obtain

$$e_0(L)=a_0(i_1-i_0)+\sum_{t=1}^{k-1} a_{i_t}(i_{t+1}-i_{t-1})+a_{i_k}(i_k-i_{k-1}).$$

For $j=0,\dots,2d$ set $\alpha_j=\min\{ a_s+a_r : s+r=j\}$ and 

$$\beta_j=\left\{\begin{array}{ll}
a_{i_t}+a_{j-i_t} & \mbox{ if } 2 i_t\leq j \leq i_t+i_{t+1}\\
a_{i_{t+1}}+a_{j-i_{t+1}} & \mbox{ if } i_t+i_{t+1}\leq j \leq 2i_{t+1} 
\end{array}
\right.
.$$
 
Since $\dim_K (R/L)=\sum_{i=0}^d a_i$ and $\dim_K (R/L^2)=\sum_{i=0}^{2d} \alpha_i$ we have
$$\begin{array}{l}
V(L)= e_0(L)-\sum_{i=0}^{2d} \alpha_i+2\sum_{i=0}^d a_i=\\ \\ 
e_0(L)-\sum_{i=0}^{2d} \beta_i+2\sum_{i=0}^d a_i+\sum_{i=0}^{2d}(\beta_i-\alpha_i)= \\ \\ 
Z+\sum_{i=0}^{2d}(\beta_i-\alpha_i)
\end{array}$$
 where 
 $$Z=a_0(i_1-i_0)+\sum_{t=1}^{k-1} a_{i_t}(i_{t+1}-i_{t-1})+a_{i_k}(i_k-i_{k-1})
 -\sum_{i=0}^{2d} \beta_i+2\sum_{i=0}^d a_i.$$

We claim that $Z=0$ identically as a linear form in the $a_i$'s. This can be checked by direct inspection. That $Z=0$ follows also from the fact that $Z$, as a linear function in the $a_i$'s, computes the deviation $V(H)$ where $H$ is any lex-segment with associated vector $\aa$ in the cone $C(\ii)$. 
Since every such lex-segment ideal has a Cohen-Macaulay associated graded ring, we have $V(H)=0$. Therefore $Z$ vanishes when evaluated at the points of $C(\ii)\cap \{\aa\in \NN^{d+1} : 0=a_0<a_1<\dots <a_d\}$ and so it is $0$ identically. 
Summing up, 
$$V(L)=\sum_{i=0}^{2d}(\beta_i-\alpha_i).$$
Now, by assumption $\ini_\aa(P)$ is Cohen-Macaulay, thus by Proposition \ref{veryimpo} $\gr_L(R)$ is Cohen-Macaulay and by Theorem \ref{redunum} $V(L)=0$. Since $\beta_i\geq \alpha_i$ for every $i$, it follows that $\beta_i=\alpha_i$ for every $i$ which in turns implies that $\aa\in \ol{C(\ii)}$. 
\end{proof} 

\begin{remark} Let $\aa=(a_0,\dots,a_d)$ be the vector associated to a lex-segment ideal $L$. Denote by $Y$ the convex hull of $\{ (i,j) : x^iy^j\in L\}$, by $V$ the set of the vertices of $Y$ and by $V'$ the set of the elements $(d-i,a_i)$ belonging to the lower boundary of $Y$. Clearly $V\subseteq V'$. Assume that $\gr_L(R)$ is Cohen-Macaulay. The proof above shows that $\aa\in \ol{C(\ii)}$ where $\{(d-j,a_j) : j\in \ii\}=V$. The same argument shows also that $\aa\in \ol{C(\ii)}$ for every $\ii$ such that $V\subseteq \{(d-j,a_j) : j\in \ii \}\subseteq V'$. 
In particular, $\aa$ belongs to $2^u$ of the cones $\ol{C(\ii)}$ where $u=\# V'-\# V$. 
\end{remark} 

The next example illustrates the remark above. 

\begin{example} Let $\aa=(0,2,4,a_3,a_4,a_5)$ with 
$$
\begin{array}{l}
4<a_3<a_4<a_5, \qquad a_3>4+(a_5-4)/3, \qquad a_4>4+2(a_5-4)/3.
\end{array}
$$
Then $V=\{(5,0), (3,4), (0,a_5)\}$ and $V'=V\cup \{(4,2)\}$. By the remark above we have that if $\ini_\aa(P)$ is Cohen-Macaulay then $\aa\in \ol {C(\ii)}\cap \ol{C(\jj)}$ with $\ii=(0,1,2,5)$ and $\jj=(0,2,5)$. In this case the Cohen-Macaulay property is equivalent to the inequalities 
$$
\begin{array}{l}
a_4\geq a_3+2, \quad a_5\geq a_4+2,\quad 2a_3\geq a_4+4,\quad 2a_4\geq a_3+a_5.
\end{array}
$$
\end{example} 

\begin{definition}
Let $\sigma\in S_d$ be a permutation. We may associate to $\sigma$ a cone 
$$C_\sigma=\{ \aa \in \QQ_{\geq 0}^{d+1} : b_{\sigma^{-1}(1)}<\dots< b_{\sigma^{-1}(d)}\}$$
where $b_i=a_i-a_{i-1}$. We call $C_\sigma$ the permutation cone associated to $\sigma$. 
\end{definition}

\begin{remark}
 As shown in the proof of Theorem \ref{moninidCM} and Remark \ref{next} each cone $C(\ii)$ contains a specific permutation cone $C_\sigma$. The permutations involved in the construction are indeed permutations avoiding the patterns ``231" and ``312". More precisely, there is a bijective correspondence between the permutations $\sigma \in S_d $ avoiding the patterns ``231" and ``312" and the cones $C(\ii)$ so that $C(\ii)\supseteq C_\sigma$. However, as we will see, the inclusion $C_\sigma\subseteq C(\ii)$ can be strict in general. The study of permutation patterns is an important subject in combinatorics, see for instance \cite{W}. 
 \end{remark}
 
The following example illustrates Theorem \ref{riassunto}. 

\begin{example} Suppose $d=6$ and take the sequence $\ii=(i_0=0,\ i_1=3,\ i_2=4,\ i_3=6)$. 
The corresponding Cohen-Macaulay initial ideal $I$ of $P$ is obtained by dividing the matrix $T_6$ in blocks (from column
$i_v+1$ to $i_{v+1}$)

$$T_6=\left( \begin{array}{cccccccc}
t_0 & t_1 & t_2 &|& t_3 &|& t_4 & t_5 \\
t_1 & t_2 & t_3 &|& t_4 &|& t_5 & t_6 
\end{array} 
\right) 
$$
and then taking anti-diagonals of minors whose columns belong to the same block, 
$$t_1^2,t_1t_2,t_2^2,t_5^2,$$ 
and main diagonals from minors whose columns belong to different blocks
$$t_0t_4, t_0t_5, t_0t_6, t_1t_4, t_1t_5, t_1t_6, t_2t_4, t_2t_5, t_2t_6, t_3t_5, t_3t_6.$$ 
The ideal $I$ is the initial ideal of $P$ with respect to every term order refining the weight
$\aa=(0,3,5,6,10,16,21)$ obtained from the ``permutation" vector $\sigma=(3,2,1|4|6,5)\in S_6$ by setting $a_0=0$ and
$a_i=\sum_{j=1}^i \sigma_j$. With respect to this term order the $2$-minors of $T_6$ are a Gr\"obner basis of $P$ but not the reduced Gr\"obner basis. The corresponding reduced Gr\"obner basis is 
$$\begin{array}{lllllll}
\underline{t_1^2}-t_0t_2, & 
\underline{t_1t_2}-t_0t_3, & 
\underline{t_2^2}-t_1t_3, & 
\underline{t_0t_4}-t_1t_3 & 
\underline{t_0t_5}-t_2t_3, \\ \\ 
\underline{t_1t_4}-t_2t_3, &
\underline{t_0t_6}-t_3^2, &
\underline{t_1t_5}-t_3^2, & 
\underline{t_2t_4}-t_3^2, & 
\underline{t_1t_6}-t_3t_4, \\ \\ 
\underline{t_2t_5}-t_3t_4, & 
\underline{t_2t_6}-t_4^2, & 
\underline{t_3t_5}-t_4^2, & 
\underline{t_3t_6}-t_4t_5, & 
\underline{t_5^2}-t_4t_6.
\end{array} 
$$
So for every vector $\aa=(a_0,a_1,\dots,a_6)\in \QQ_{\geq 0}^{7}$ satisfying the following system of linear inequalities
$$\begin{array}{lllll}
2a_1>a_0+a_2*& 
a_1+a_2>a_0+a_3 & 
2a_2>a_1+a_3* & 
a_0+a_4>a_1+a_3* \\
a_0+a_5>a_2+a_3 & 
a_1+a_4>a_2+a_3 &
a_0+a_6>2a_3 &
a_1+a_5>2a_3 \\
a_2+a_4>2a_3 & 
a_1+a_6>a_3+a_4 &
a_2+a_5>a_3+a_4 & 
a_2+a_6>2a_4 \\
a_3+a_5>2a_4 &
a_3+a_6>a_4+a_5* & 
2a_5>a_4+a_6 *
\end{array} 
$$
we have $\ini_\aa(P)=I$. 
The $15$ linear homogeneous inequalities above define the open Cohen-Macaulay cone $C(\ii)$. The description is however far from being minimal. The inequalities marked with the star $*$ are indeed sufficient to describe $C(\ii)$. In terms of $b_i=a_i-a_{i-1}$ the inequalities can be described by 
$b_3<b_2<b_1<b_4<b_6<b_5,$ that is, $C(\ii)=C_{\sigma}$. 
\end{example} 
 
\begin{remark} \begin{itemize} 
\item[(1)] There exist Cohen-Macaulay ideals of dimension $2$ with minimal multiplicity and without Cohen-Macaulay initial monomial ideals in the given coordinates. For instance, let $J$ be the ideal of $K[t_0,\dots,t_4]$ generated by the $2$-minors of the matrix 
$$\left( \begin{array}{ccccc}
t_0& t_2& t_4-t_0& 0 \\
t_1& t_3& 0 & t_4+t_0
\end{array}\right).
$$
Then $J$ has the expected codimension and hence it is $2$-dimensional Cohen-Macaulay with minimal multiplicity. No monomial initial ideal of $J$ is quadratic since the degree $2$ part of every monomial initial ideal has codimension $2$. Hence no monomial initial ideal of $J$ is
Cohen-Macaulay. This example shows that Theorem \ref{nonmon} does not hold for $2$-dimensional binomial Cohen-Macaulay ideals with minimal multiplicity. 

\item[(2)] $\ini_\aa(P)$ can be monomial, quadratic and
non-Cohen-Macaulay. For example, for $d=4$ the ideal $I=(t_3^2, t_2t_3, t_1t_3, t_0t_4, t_0t_3, t_1^2)$ is a non-Cohen-Macaulay (indeed non-pure) monomial initial ideal of $P$. The corresponding cone is described in terms of $b_i=a_i-a_{i-1}$ by the inequalities 
$b_3>b_1>b_2>b_4$ and $b_3+b_4>b_1+b_2$. 

\item[(3)] $\ini_\aa(P)$ can be quadratic with linear $1$-syzygies and not Cohen-Macaulay. 
For instance, with $d=7$ the ideal generated by 
$$\begin{array}{l}
t_2^2,\ t_4^2,\ t_6^2,\ t_1t_2,\ t_0t_4,\ t_0t_5,\ t_1t_4,\ t_0t_6,\ t_1t_5,\ t_2t_4,\ t_0t_7,\ t_1t_6,\ 
t_3t_4,\ \\ t_1t_7,\ t_2t_6,\ t_3t_6,\ t_3t_7,\ t_4t_6,\ t_1^2+t_0t_2,\ -t_4t_5+t_2t_7,\ -t_5t_6+t_4t_7
\end{array} $$
 is an initial ideal of $P$ with linear $1$-syzygies and a non-linear $2$-syzygy. 
 
 \item[(4)] We do not know an example as the one in (3) if we further assume that $\ini_\aa(P)$ is a monomial ideal. 
Note however that $2$-dimensional non-Cohen-Macaulay quadratic monomial ideals with a short h-vector and linear $1$-syzygies exist, 
for example $(t_1t_3, t_1t_5, t_0t_2, t_2t_5, t_0t_3, t_2^2, t_2t_4, t_2t_3, t_0t_4, t_4t_5)$. 
\end{itemize} 
\end{remark} 

\begin{remark}\label{Rekha} As Rekha Thomas pointed out to us, 
one can deduce from results in \cite{HT,OT} that $P$ has exactly one Cohen-Macaulay monomial initial ideal for each regular triangulation of the underlying point configuration $A$. In \cite[Theorem 5.5(ii)]{HT} it is proved that for every regular triangulation of $A$ there exists exactly one initial ideal having no embedded primes (they are called Gomory initial ideals in the paper). In \cite{OT} it is proved that every Gomory initial ideal coming from a $\Delta$-normal configuration is Cohen-Macaulay. Since every triangulation of $A$ is $\Delta$-normal, one can conclude that the Gomory ideals of $P$ are indeed Cohen-Macaulay. Hence these results imply that $P$ has exactly $2^{d-1}$ Cohen-Macaulay monomial initial ideals. 
\end{remark}

\section{Contracted ideals whose associated graded ring is Cohen-Macaulay}\label{contraCM}

In this section we use the results of Section \ref{CM for P} to solve the problem (2) mentioned in the introduction. 
Since the ideal $P$ is homogeneous with respect to the vectors $(1,1,\dots,1)$ and $(0,1,2,\dots,d)$ of $\QQ^{d+1}$ each cone of the Gr\"obner fan of $P$ is determined by its intersection with 
$$W_d=\{(a_0,a_1,\dots, a_d)\in \NN^{d+1} : 0=a_0<a_1<\dots<a_d\}.$$
As explained in Section \ref{contra2}, $W_d$ parametrizes the lex-segment ideals of initial degree $d$. 
For a given $d\in \NN, d>0$ we set 
$$CM_d=\{ \aa\in \QQ_{\geq 0}^{d+1} : \ini_{\aa}(P) \mbox{ is Cohen-Macaulay}\}$$ 
the ``Cohen-Macaulay region" of the Gr\"obner fan of $P$. According to 
\ref{nonmon} we have 
$$CM_d=\cup_{\ii} \ol{C(\ii)}$$
where the union is extended to all the $2^{d-1}$ sequences 
$$\ii=(0=i_0<i_1<\dots<i_k=d).$$

\begin{theorem} 
\label{togo}
Let $d_1,\dots,d_s$ be positive integers and $\aa_1,\dots, \aa_s$ be vectors such that $\aa_i\in W_{d_i}$. Let $\ell_1,\dots,\ell_s,z$ be linear forms in $R=K[x,y]$ such that each pair of them is linearly independent. For every $i=1,\dots,s$ consider the lex-segment ideals $L_i$ associated to $\aa_i$ with respect to $\ell_i,z$, that is, 
$L_i=(\ell_i^{d_i-j}z^{a_{ij}} : j=0,\dots, d_i)$. Set $I=L_1\cdots L_s$. We have: 

\begin{itemize} 
\item[(1)] $I$ is contracted and every homogeneous contracted ideal in $R=K[x,y]$ arises in this way. 
\item[(2)] $\gr_I(R)$ is Cohen-Macaulay iff $\aa_i\in CM_{d_i}$ for all $i=1,\dots,s$. 
\end{itemize}
\end{theorem} 
\begin{proof} 
(1) is a restatement of Zariski's factorization theorem for contracted ideals. Statement (2) follows from Theorem \ref{ZaGr}, Proposition \ref{veryimpo} and Theorem \ref{nonmon}. 
\end{proof}

Theorem \ref{togo} can be generalized as follows. 
 
\begin{theorem} 
\label{ritogo}
 Let $I\subset K[x,y]$ be a monomial ideal (not necessarily contracted) and let $\aa=(a_0,\dots,a_d)$ be its associated sequence. 
Then $\gr_I(R)$ is Cohen-Macaulay if and only if $\aa\in CM_d$. 
\end{theorem} 
\begin{proof} 
If $\aa$ is strictly increasing, then $I$ is a lex-segment ideal. Hence $I$ is contracted and the statement is a special case of Theorem \ref{togo}. 
If $\aa$ is not strictly increasing, then we set $\aa'=\aa+(0,1,\dots,d)$ and let $L$ be the monomial ideal associated to $\aa'$. Since $\aa'$ is strictly increasing, $L$ is a lex-segment ideal. 
The cones $C(\ii)$ are described by inequalities that are homogeneous with respect to $(0,1,\dots,d)$. Therefore $\aa$ belongs to $CM_d$ if and only if $\aa'$ does it. 
By construction, $I$ is the quadratic transform of the contracted ideal $L$ in the sense of \cite[Sect.3]{CDJR}. Further we know that $\depth \gr_I(R)=\depth \gr_L(R)$, \cite[Thm.3.12]{CDJR}. Summing up, we have: $\gr_I(R)$ is Cohen-Macaulay iff $\gr_L(R)$ is Cohen-Macaulay iff $\aa'\in CM_d$ iff $\aa\in CM_d$. 
\end{proof} 

\begin{remark} 
\begin{itemize} 
\item[(1)] In $K[x,y]$ denote by $C$ the class of contracted ideals, by $C_1$ the class of the ideals in $C$ with Cohen-Macaulay associated graded ring and by $C_2$ the class of integrally closed ideals. We have $C\supset C_1\supset C_2$. 
One knows that $C$ and $C_2$ are closed under product. On the other hand $C_1$ is not: the lex-segment ideals associated to the sequences $(0, 4, 6, 7)$ and $(0, 2)$ belong to $C_1$ and their product does not. 
However, $C_1$ is closed under powers: if $I\in C_1$ then $I^k\in C_1$. This can be seen, for instance, by looking at the Hilbert function of $I$. Furthermore we will show in Section \ref{bigcone} that a certain subset of $C_1$ is closed under product. 
\item[(2)] For a lex-segment ideal $L$ in $K[x,y]$ we have seen that $\gr_L(R)$ and $F(L)$ have the same depth. We believe that $\gr_I(R)$ and $F(I)$ have the same depth for every contracted ideal $I$. In \cite[Thm.3.7, Cor.3.8]{DRV} D'Cruz, Raghavan and Verma proved that the Cohen-Macaulayness of $\gr_I(R)$ is equivalent to that of $F(I)$. 
Note however that for a monomial ideal $I$ the rings $\gr_I(R)$ and $F(I)$ might have different depth. For instance, for the ideal $I$ associated to $(0,2,2,3)$ one has $\depth \gr_I(R)=1$ and $\depth F(I)=2$. 
\end{itemize} 
\end{remark}

\begin{remark} Two of the cones of the Cohen-Macaulay region $CM_d$ are special as they correspond to opposite extreme selections: 
\begin{itemize}
\item[(1)] If $\ii=(0,1,2,\dots,d)$, then the closed cone $\ol{C(\ii)}$ is described by the inequality system 
$$a_i+a_j\geq a_u+a_v $$ 
with $u=\lfloor (i+j)/2 \rfloor , \ v=\lceil (i+j)/2$ for every $i,j$
 or, equivalently, by
$$b_{i+1}\geq b_i$$
 for every $i=1,\dots,d-1$. 
In other words, $C(\ii)$ equals its permutation cone $C_{\id}$, where $\id\in S_d$ is the identity permutation. 
 In this case the initial ideal of $P$ is $(t_it_j : j-i>1)$ and it can be realized by the lex-order with $t_0<t_1<\dots<t_d$ or by the lex-order with $t_0>t_1>\dots>t_d$. This is the only radical monomial initial ideal of $P$. The points in $W_d\cap \ol{C(\ii)}$ correspond to integrally closed lex-segment ideals. Indeed, they are the products of $d$ complete intersections of type $(x,y^u)$. 
\item[(2)] 
If $\ii=(0,d)$ then the closed cone $\ol{C(\ii)}$ is described by inequality system 
$$a_i+a_j\geq a_0+a_{i+j}$$
 if $i+j\leq d$ and 
 $$a_i+a_j\geq a_d+a_{i+j-d}$$
 if $i+j\geq d$. It can be realized by the revlex order with $t_0<t_1<\dots<t_d$ or by the
revlex order with $t_0>t_1>\dots>t_d$. The corresponding initial ideal of $P$ is $(t_1,\dots,t_{d-1})^2$. 
The lex-segment ideals $L$ belonging to the cone are characterized by the fact that 
$L^2=(x^d,y^{a_d})L$, that is, they are exactly the lex-segment ideals with a monomial minimal reduction and reduction number $1$. It is not difficult to show that the simple homogeneous integrally closed ideals of $K[x,y]$ are exactly the ideals of the form 
$\ol{(x^d,y^c)}$ with $\GCD(d,c)=1$. In other words, $\ol{C(\ii)}$ contains (the exponent vectors of) all the simple integrally closed ideals of order $d$. 
The associated permutation cone is $C_\sigma$ with $\sigma=(d,d-1,\dots,1)$. 
For $d\leq 3$ one has $C(\ii)=C_\sigma$. For $d=4$ one has $C(\ii)\supsetneq C_\sigma$ and 
$\ol{C(\ii)}=\ol{C_\sigma} \cup \ol{C_\tau}$ with $\tau=(4,2,3,1)$. For $d>4$ the cone $\ol{C(\ii)}$ is not the union of the 
closure of the permutation cones it contains. For $d=5$, for example, the cone $\ol{C(\ii)}$ is described by the inequalities
$$b_1\geq b_i \geq b_5 \mbox{ with } i=2,3,4, \quad b_1+b_2\geq b_3+b_4, \quad b_2+b_3\geq b_4+b_5$$ 
and hence it intersects but it does not contain the cone associated with the permutation $(5,2,4,3,1)$. 
\item[(3)] Apart from the example discussed in (1), the other Cohen-Macaulay monomial initial ideals of $P$ arising from lex orders are exactly those associated to sequences $\ii=(0,1,\dots,\widehat{j},\dots,d)$ for some $0<j<d$. Apart from the example discussed in (2), the other Cohen-Macaulay monomial initial ideals of $P$ arising from revlex orders are exactly those associated to sequences $\ii=(0,j,d)$ for some $0<j<d$. Therefore, starting from $d=5$, there are Cohen-Macaulay monomial initial ideals of $P$ not coming from lex or revlex orders. For instance, the initial ideal associated to $\ii=(0,1,4,5)$ is such an example. 
\end{itemize} 
\end{remark}

 \section{Describing the Hilbert series of $\gr_L(R)$}\label{HF}
Let $L$ be a lex-segment ideal in $R=K[x,y]$ with associated $a$-sequence $\aa=(a_0,a_1,\dots,a_d)$. We have seen in the proof of Theorem \ref{nonmon} that the multiplicity $e_0(L)$ can be expressed as a linear function in the $a_i$'s. In terms of initial ideals of $P$, that assertion can be rephrased as follows: let $I$ be a monomial initial ideal of $P$ and let $C_I$ be the corresponding closed cone in the Gr\"obner fan of $P$ 
 $$C_I=\{ \aa\in \QQ_{\geq 0}^{d+1} : \ini_\tau(\ini_\aa(P))=I \}$$
 where $\tau$ is a given term order such that $\ini_\tau(P)=I$. 
Let $\ii=(0=i_0<i_1<\dots<i_k=d)$ the set of the integers $0\leq j\leq d$ such that $t_j\not\in \sqrt I$. Then $\sqrt I=(t_j : j\not\in \ii)+(t_{i_v}t_{i_r} : r-v>1)$. Consider the following linear form in $\ZZ[A_0,\dots,A_d]$ 
$$e^I_0=A_0 (i_1-i_0)+\sum_{t=1}^{k-1} A_{i_t} (i_{t+1}-i_{t-1}) + A_{i_k}(i_k-i_{k-1}) $$
where the $A_i$ are variables. 
 For every lex-segment ideal $L$ with associated sequence $\aa$ belonging to $C_I$ one has that $e_0(L)$ is equal to $e^I_0$ evaluated at $A=\aa$. 
 
So the ``same" formula holds in all the cones of the Gr\"obner fan associated with the same radical, i.e. in all the cones whose union form a maximal cone of the secondary fan \cite[pag.71]{St}. We establish now similar formulas for the Hilbert function $H^1(L,k)$ and the h-polynomial of $\gr_L(R)$.

To this end, consider $S=K[t_0, t_1,\dots , t_d]$ equipped with its natural $\ZZ^{d+1}$-graded structure. The quotient $S/I$ is $\ZZ^{d+1}$-graded and we denote by 
$H_{S/I}(\underline{t})$ its $\ZZ^{d+1}$ graded Hilbert series, namely
$$H_{S/I}(\underline{t})=\sum_{\alpha\in \NN^{d+1} } \dim [S/I]_\alpha t^\alpha=\sum_{t^\alpha\not\in I} t^\alpha$$
where $t^\alpha=t_0^{\alpha_0}\cdots t_d^{\alpha_d}$. 
The key observation is contained in the following lemma. 

\begin{lemma}\label{ancoratu}
Let $L$ be a lex-segment ideal with associated vector $\aa$ belonging to $C_I$. For $k\in \NN$ set $M_k(I)=\{ \alpha\in \NN^{d+1} : t^\alpha\not\in I, |\alpha|=k \}$. Denote by $\sum M_k(I)$ the sum of the vectors in $M_k(I)$. By construction $\sum M_k(I) \in \NN^{d+1}$ and 
$$H^1(L,k-1)=\aa \cdot \sum M_k(I)$$
for all $k$. 
\end{lemma} 

\begin{proof} 
Set $C_k=\aa \cdot \sum M_k(I) $. 
Writing $t^\alpha$ as $t_{j_1}\cdots t_{j_k}$, we may rewrite $C_k$ as the sum $a_{j_1}+\cdots +a_{j_{k}}$ over all monomials $t_{j_1}\cdots t_{j_{k}} \not\in I$. By
construction $a_{j_1}+\cdots +a_{j_k}=\min\{ a_{i_1}+a_{i_2}+\dots+ a_{i_k} :
i_1+i_2+\dots+i_k=j_1+j_2+\dots+j_k\}$ if and only if $t_{j_1}\cdots t_{j_{k}} \not\in I$. Therefore $C_k$ is the sum over all $v$, $0\leq v\leq kd$ of $\min\{ a_{i_1}+a_{i_2}+\dots+ a_{i_k} : i_1+i_2+\dots+i_k=v\}$. But this is exactly $H^1(L,k-1)$, see Lemma \ref{formH}. 
\end{proof} 

In terms of Hilbert series Lemma \ref{ancoratu} can be rewritten as in the following lemma. 

 \begin{lemma}\label{RuiCosta} Let $L$ be a monomial ideal with associated sequence $\aa$ belonging to $C_I$. Then 
$$H^1_L(z)=\aa \cdot \nabla H_{S/I}(\underline{t})_{t_i=z} $$
where $\nabla =(\partial/\partial t_0, \dots, \partial/\partial t_d)$ is the gradient operator. 
 \end{lemma} 

\begin{remark}\label{LuisFigo} \begin{itemize}
\item[(1)] The series $H_{S/I}(\underline{t})$ is rational and it can be described in terms of the multigraded Betti numbers 
$\beta_{i,\alpha}(S/I)=\dim_K \Tor_i^S(S/I,K)_\alpha:$ 
$$H_{S/I}(\underline{t})=\frac{\sum_{i, \alpha} (-1)^i t^{\alpha} \beta_{i,\alpha}(S/I)}{\Pi_{i=0}^d (1-t_i)} .$$
\item[(2)]
A rational expression of $H_{S/I}(\underline{t})$ can be computed also from a Stanley decomposition of $S/I$. For a monomial ideal $I$ a Stanley decomposition of $S/I$ is a finite set $\Omega$ of pairs $(\sigma,\tau)$ where $\sigma \in \NN^{d+1}$ and $\tau \subseteq \{0,\dots,d\}$ which induces a decomposition
$$S/I=\bigoplus_{(\sigma,\tau)\in \Omega}\ t^\sigma K[ t_i : i\in \tau]$$
as a $K$-vector space. Stanley decompositions always exist but they are far from being unique. There are algorithms to compute them, see \cite{MS} for more. 
For every Stanley decomposition $\Omega$ of $S/I$ clearly one has: 
$$H_{S/I}(\underline{t})=\sum_{(\sigma,\tau)\in \Omega} \frac{ t^\sigma}{\Pi_{i\in \tau} (1-t_i)}.$$ 
\end{itemize}
\end{remark}

Combining Lemma \ref{ancoratu}, Remark \ref{LuisFigo} with Lemma \ref{RuiCosta} we obtain the following result. 

\begin{corollary}\label{Maniche} For every cone $C_I$ and for a monomial ideal $L$ whose associated sequence $\aa$ belongs to $C_I$ we have 
$$H^1_L(z)= |\aa|\frac{1+(d-1)z}{(1-z)^3} + \aa \cdot \sum_{i\geq 1} (-1)^i \sum_{\alpha} \beta_{i,\alpha}(S/I)
\frac{ z^{|\alpha|-1}}{(1-z)^{d+1}} \alpha,$$
where the $\beta_{i,\alpha}(S/I) $ are the multigraded Betti numbers of $S/I$. Moreover
$$H^1_L(z)=\aa \cdot \sum_{(\sigma, \tau)\in \Omega} 
\frac{ z^{|\sigma|-1}}{(1-z)^{|\tau|+1}}
(z\tau+(1-z)\sigma),$$
where $\Omega$ is a Stanley decomposition of $S/I$ and where we have identified the subset $\tau$ with the corresponding $0/1$-vector. 
\end{corollary} 

 The next proposition summarizes what we have proved so far concerning formulas for $h^I$ and related invariants. 

\begin{proposition}\label{Super} 
 Given a cone $C_I$ of the Gr\"obner fan of $P$ there are polynomials $h^I \in \ZZ[A,z]$ and $Q^I,Q_1^I \in \QQ[A,z]$ linear in the variables $A=A_0,\dots,A_d$ and without constant term, such that: 
\begin{itemize} 
\item[(1)] For every monomial ideal $L$ with associated sequence $\aa$ belonging to $C_I$ the polynomial $h^I$ evaluated at $A=\aa$ equals the $h$-polynomial of $L$, $Q^I$ evaluated at $A=\aa$ equals the Hilbert polynomial $P_L$ of $L$, and $Q_1^I$ evaluated at $A=\aa$ equals the Hilbert polynomial $P^1_L$ of $L$.
\item[(2)] $h^I, Q^I$ and $Q_1^I$ can be expressed in terms of the multigraded Betti numbers of $I$. They can also be expressed in terms of a Stanley decomposition of $S/I$. 
\item[(3)] In particular, 
$$h^I= |A|(1+(d-1)z) + A \sum_{i\geq 1} (-1)^i \sum_{\alpha} \beta_{i,\alpha}(S/I)
\frac{z^{|\alpha|-1}}{(1-z)^{d-2}} \alpha, $$
where $\beta_{i,\alpha}(S/I)$ are the multigraded Betti numbers of $S/I$ and 
$$h^I=A \cdot \sum_{(\sigma, \tau)\in \Omega} z^{|\sigma|-1}(1-z)^{2-|\tau|} 
\left( z \tau+(1-z)\sigma \right),
$$
where $\Omega$ is a Stanley decomposition of $S/I$. 
Explicit expressions for $Q^I$ and $Q_1^I$ can be obtained from that of $h^I$. 
\end{itemize} 
\end{proposition} 

 Similarly one has expressions for the Hilbert coefficients $e_i^I$ as a linear function in the variables $A$. 
Now we discuss the dependence of the polynomials $h^I$ and $Q_1^I$ on $I$. 

\begin{proposition}\label{stessoh}
Let $I,J$ be monomial initial ideals of $P$. Then 
\begin{itemize} 
\item[(1)] $e_{0}^I=e_{0}^J$ iff $\sqrt{I}=\sqrt{J}$.
\item[(2)] $h^I=h^J$ iff $I=J$. 
\item[(3)] $Q_1^I=Q_1^J$ iff $I$ and $J$ have the same saturation, equivalently, they coincide from a certain degree on. 
\end{itemize} 
\end{proposition} 

\begin{proof} Denote by $A$ the vector of variables $(A_0,\dots,A_d)$. We have discussed already the fact that the formula for the multiplicity $e_{0}^I$ identifies and it is identified by the radical of $I$. For statement (2), we have already seen that the coefficient $C_k$ of $z^k$ in the series $h^I/(1-z)^3$ is exactly $A\cdot \sum M_{k+1}(I)$. 
Hence $h^I=h^J$ holds if and only if $A \cdot \sum M_{k}(I) =A \cdot \sum M_{k}(J)$ for all $k$, that is, $\sum M_{k}(I)= \sum M_{k}(J)$ as vectors for every $k$. By virtue of \cite[Corollary 2.7]{St}, we conclude that $h^I=h^J$ implies $I=J$. For (3) one just applies the same argument to all large degrees. 
\end{proof} 
For an ideal $I$ of dimension $v$ we denote by $I^{top}$ the component of dimension $v$ of $I$, that is, the intersection of the primary components of $I$ of dimension $v$. 

\begin{remark} \label{toppo}
Let $I,J$ monomial initial ideals of $P$. In terms of $M_k(I)$ the condition 
$Q^I=Q^J$ is equivalent to $\sum M_k(I)-\sum M_{k-1}(I)=\sum M_k(J)-\sum M_{k-1}(J)$ for all $k>>0$. 
There is some computational evidence that $Q^I=Q^J$ could be equivalent to $I^{top}=J^{top}$. This is related with the hypergeometric Gr\"obner fan of $P$, see \cite[Section 3.3]{SST}. In particular Example 3.3.7 in 
\cite{SST} discusses the secondary fan, the hypergeometric fan and Gr\"obner fan of $P$ for $d=4$.
 \end{remark} 
 
\section{The big Cohen-Macaulay cone}\label{bigcone}

Starting from $d=3$, the Cohen-Macaulay region $CM_d$ it is not a cone (i.e. it is not convex), see Section \ref{smalld} for examples. However, a bunch of the cones $C(\ii)$ get together to form a big cone. 

\begin{proposition}\label{bCM}
Let $B_d=\cup_{\ii} \ol{C(\ii)}$ where the union is extended to all the sequences $\ii=\{0=i_0<i_1<\dots<i_k=d\}$ such that $i_{v}-i_{v-1}\leq 2$ for all $v=1,\dots,k$. 
Then $B_d$ is the closed cone described, in terms of the $b_i$'s, by the inequalities $b_j\leq b_{j+2}$ for all $j=1,\dots,d-2$. 
\end{proposition}

\begin{proof} Let $B'$ be the cone described by the inequalities $b_j\leq b_{j+2}$ for all $j=1,\dots,d-2$. We have to show that $B_d=B'$. For the inclusion $\subseteq$, let $\aa\in B_d$ and $b_j=a_j-a_{j-1}$. Then $\aa\in \ol{C(\ii)}$ for a sequence $\ii=\{0=i_0<i_1<\dots<i_k=d\}$ such that $i_{v}-i_{v-1}\leq 2$ for all $v=1,\dots,k$. 
For every $j$, $1\leq j\leq d-2$, at least one among $j$ and $j+1$ is in $\ii$. We distinguish two cases: 

Case (1) $j\in \ii$, say $j=i_v$. Then set $s=j-1$ and $r=j+2$. We have $2i_v\leq s+r\leq i_v+i_{v+1}$ and so, by Theorem \ref{riassunto} (5), $a_s+a_r\geq a_{i_v}+a_{s+r-i_v}$ is one of the defining inequalities of $\ol{C(\ii)}$. Explicitly, it says $a_{j-1}+a_{j+2}\geq a_{j}+a_{j+1}$, i.e., $b_j\leq b_{j+2}$. 

Case (2) $j+1\in \ii$, say $j+1=i_{v+1}$. Then set $s=j-1$ and $r=j+2$. We have $i_v+i_{v+1}\leq s+r\leq 2i_{v+1}$ and so, by Theorem \ref{riassunto} (5) $a_s+a_r\geq a_{i_{v+1}}+a_{s+r-i_{v+1}}$ is one of the defining inequalities of $\ol{C(\ii)}$. Explicitly, $a_{j-1}+a_{j+2}\geq a_{j}+a_{j+1}$, i.e., $b_j\leq b_{j+2}$.

For the inclusion $\supseteq$, let $\aa\in B'$ and $b_i=a_i-a_{i-1}$. Set $U=\{ j : 1\leq j\leq d-1 : b_j>b_{j+1}\}$. Since $b_j\leq b_{j+2}$ for all $j$, $U$ does not contain pairs of consecutive numbers. Set 
$$\ii=\{0,1,\dots,d\}\setminus U=(0=i_0<\dots<i_k=d).$$
Note that for all $0\leq r,s\leq d$ such that $s-r\geq 2$ one has 
$b_s+b_{s-1}\geq b_{r+1}+b_r$
and hence 
$$a_s+a_r\geq a_{s-2}+a_{r+2}.$$
This fact, together with the definition of $U$, implies that for all 
$0\leq r,s\leq d$ one has: 

$$a_s+a_r\geq 
\left\{
\begin{array}{ll}
a_j+a_{j+1} & \mbox{ if } s+r=2j+1\\
2a_j & \mbox{ if } s+r=2j \mbox{ and } j\in \ii\\
a_{j-1}+a_{j+1} & \mbox{ if } s+r=2j \mbox{ and } j\not\in \ii
\end{array}
\right.
$$
Using this information one proves directly that $\aa$ satisfies the 
inequalities of Theorem \ref{riassunto} (5) defining $\ol{C(\ii)}$. 
\end{proof}

\begin{remark} 
\begin{itemize}
\item[(1)] The number $f_d$ of the cones $\ol{C(\ii)}$ appearing in the description of $B_d$ satisfies the recursion $f_d=f_{d-1}+f_{d-2}$ with $f_1=1$ and $f_2=2$. Hence $f_d$ is the $(d+1)$-th Fibonacci number. 
\item[(2)] One also has $B_d=\cup \ol{C_\sigma}$ where $\sigma\in S_d$ satisfies $\sigma(j)<\sigma(j+2)$ for $j=1,\dots,d-2$. There are $\binom{d}{\lfloor d/2 \rfloor}$ such permutations. 
\item[(3)] Indeed, each cone $\ol{C(\ii)}$ appearing in the description of $B_d$ is the union of 
permutation cones $\ol{C_\sigma}$. Precisely, the permutations involved are those $\sigma\in S_d$ such that $\sigma(j)<\sigma(j+2)$ for all $j=1,\dots,d-2$ and such that $\sigma(j)>\sigma(j+1)$ iff $j\not\in \ii$. The number of these permutations, say $n(\ii)$, is a product of Catalan numbers. Recall that the $n$-th Catalan number is $c(n)=(n+1)^{-1}\binom{2n}{n}$. Decompose $\{1,\dots,d\}\setminus \ii$ as a disjoint union $\cup_{i=1}^t V_i$ where the $V_i$ are of the form $\{a,a+2,\dots\}$ and are maximal. Then $n(\ii)=c(|V_1|)\cdots c(|V_t|)$. For instance, if $\ii=(0,2,3,5,7,9,10,12,14)$ then $\{1,\dots,14\}\setminus \ii=\{1\}\cup \{4,6,8\}\cup \{11,13\}$ and hence $\ol{C(\ii)}$ is the union of $c(1)c(3)c(2)=10$ permutation cones. 
\item[(4)] The family $B_d$ with $d\in \NN$ is closed under multiplication, that is, if $\aa\in B_d$ and $\aa'\in B_e$ and $\cc=\aa\cdot \aa'$, then $\cc \in B_{d+e}$. Set $\cc=(c_0,c_1,\dots,c_{d+e})$. To show that $\cc\in B_{d+e}$ one has to prove that $c_{j+2}-c_{j+1}\geq c_{j}-c_{j-1}$, that is, 
$$c_{j+2}+c_{j-1}\geq c_{j+1}+c_{j}.$$ 
By definition, $c_{j+2}=a_v+a'_u$ with $v+u=j+2$ and $c_{j-1}=a_w+a'_z$ with $w+z=j-1$. 
Since $(v-w)+(u-z)=(v+u)-(w+z)=3$ we may assume that $v-w\geq 2$. 
Then $a_v+a_w\geq a_{v-2}+a_{w+2}$ and hence 
$c_{j+2}+c_{j-1}=a_v+a'_u+a_w+a'_z\geq a_{v-2}+a'_u+a_{w+2}+a'_z\geq c_j+c_{j+1}$. 
\end{itemize} 
\end{remark}

\section{Examples with small $d$}\label{smalld}

In this section we describe, for small $d$, the Gr\"obner fan, the Cohen-Macaulay region, and give formulas for the Hilbert series associated to the various cones. For simplicity, the cones will be described in terms of $b_1,\dots,b_d$ where $b_i=a_i-a_{i-1}$. \medskip 

For $d=1$, there is no much to say. The ideal $P$ is $0$, $CM_1=\QQ_{\geq 0}^2$ and $CM_1\cap W_1=\{ (0,a)\in \NN^2 : a>0\}$. For $d=2$ the Gr\"obner fan has two maximal cones, both Cohen-Macaulay. The lex cone $C(0,1,2)$ described by $b_1\leq b_2$ and the revlex cone $C(0,2)$ described by $b_1\geq b_2$. 

For $d=3$ the Gr\"obner fan has $8$ maximal cones, $4$ of them are Cohen-Macaulay and $4$ have depth $1$. We show below the cones. Each table shows:

\begin{itemize} 
\item[(1)] an initial ideal $I$ of $P$,
\item[(2)] the linear inequalities defining the corresponding cone in the Gr\"obner fan,
 \item[(3)] the coefficients of the $h$-vector of $\gr_L(R)$ for the ideal $L$
 corresponding to points $(a_0,a_1,a_2,a_3)$ in the cone. 
 \end{itemize} 
 
 The expressions of the $h$-vectors have been computed using Stanley decompositions and the formula in Lemma \ref{Maniche}. The Stanley decompositions have been computed by the algorithm presented in \cite{MS}. Cones (a),(b),(c) and (d) are Cohen-Macaulay cones. In particular (a) is the lex cone and (d) is the revlex cone. The union of (a),(b) and (c) is the big cone $B_3$ and it is defined by $b_1\leq b_3$. The revlex cone (d) is isolated, it intersects $B_3$ only at $b_1=b_2=b_3$. In particular the Cohen-Macaulay region is not a cone. 
 
 $$
\begin{array}{rl}
\boxed{\begin{array}{lllll} 
(a) \ (t_1t_3, t_0t_3, t_0t_2),\\ 
b_1\leq b_2\leq b_3\\
 (h_0)\ \ a_0 + a_1 + a_2 + a_3 \\
 (h_1)\ \ a_1 + a_2 
 \end{array} } &
\boxed{\begin{array}{lllll}
(b)\ (t_1t_3, t_0t_3, t_1^2)\\
 b_2\leq b_1\leq b_3\\
 (h_0) \ \ a_0 + a_1 + a_2 + a_3 \\ 
 (h_1)\ \ a_0 - a_1 + 2a_2 
 \end{array} } \\ \\
\boxed{\begin{array}{lllll}
 (c)\ (t_2^2, t_0t_3, t_0t_2),\\ 
 b_1\leq b_3\leq b_2,\\
 (h_0)\ \ a_0 + a_1 + a_2 + a_3\\
 (h_1)\ \ 2a_1 - a_2 + a_3 
 \end{array} } & 
\boxed{\begin{array}{lllll}
 (d)\ (t_2^2, t_1t_2, t_1^2),\\ 
 b_3\leq b_2\leq b_1, \\
 (h_0)\ \ a_0 + a_1 + a_2 + a_3\\ 
 (h_1)\ \ 2a_0 - a_1 - a_2 + 2a_3 
\end{array} } \\ \\ 
\end{array} $$
The non-Cohen-Macaulay cones are: 
$$
\begin{array}{rl}
\boxed{\begin{array}{lllll}
 (e)\ (t_2^2, t_1t_2, t_0t_2, t_0^2t_3),\\ 
 b_3\leq b_1\leq b_2 \mbox{ and } b_3+b_2\geq 2b_1, \\ 
 (h_0) \ \ a_0 + a_1 + a_2 + a_3 \\
 (h_1) \ \ a_0 + a_1 - 2a_2 + 2a_3 \\
 (h_2) \ \ -a_0 + a_1 + a_2 - a_3 
 \end{array} } & 
\boxed{\begin{array}{lllll}
(f)\ (t_1t_3, t_1t_2, t_1^2, t_0t_3^2),\\ 
b_2\leq b_3\leq b_1 \mbox{ and } b_1+b_2\leq 2b_3, \\
(h_0)\ \ a_0 + a_1 + a_2 + a_3 \\
(h_1) \ \ 2a_0 - 2a_1 + a_2 + a_3 \\
(h_2) \ \ -a_0 + a_1 + a_2 - a_3 
\end{array} }
\\ \\ 
\boxed{\begin{array}{lllll}
(g)\ (t_1t_3, t_1t_2, t_1^2, t_2^3),\\ 
 b_2\leq b_3\leq b_1 \mbox{ and } b_1+b_2\geq 2b_3, \\
 (h_0)\ \ a_0 + a_1 + a_2 + a_3 \\
 (h_1)\ \ 2a_0 - 2a_1 + a_2 + a_3 \\
 (h_2)\ \ a_1 - 2a_2 + a_3 
\end{array} } & 
\boxed{\begin{array}{lllll}
(h)\ (t_2^2, t_1t_2, t_0t_2, t_1^3),\\
 b_3\leq b_1\leq b_2 \mbox{ and } b_3+b_2\leq 2b_1,\\
 (h_0)\ \ a_0 + a_1 + a_2 + a_3 \\
 (h_1)\ \ a_0 + a_1 - 2a_2 + 2a_3 \\
 (h_2)\ \ a_0 - 2a_1 + a_2 
 \end{array} }
 \end{array} $$

For $d=4$ there are $42$ cones of the Gr\"obner fan, $10$ of them have depth $0$, $24$ have depth $1$ and $8$ are Cohen-Macaulay. The big Cohen-Macaulay cone is the union of $5$ of the $8$ Cohen-Macaulay cones. The remaining $3$ are isolated. The following example illustrates Proposition \ref{stessoh} and Remark \ref{toppo}. The ideals $I,J$ below are non-Cohen-Macaulay initial ideals of $P$. They satisfy $Q^I=Q^J$ and $Q_1^I\neq Q_1^J$. We display the ideals and the formulas for the coefficients $e_0,e_1, e_2$ that have been computed via Stanley decompositions. 
 
 $$
\begin{array}{ll}
\boxed{\begin{array}{lllll}
&I& (t_1t_3, t_1t_2, t_0t_2, t_3^3, t_1^2t_4, t_1^3, t_2t_4, t_2t_3, t_2^2) \\
& (e_0) & 4a_0 + 4a_4\\
& (e_1) & 3a_0 - a_1 - 3a_3 + 4a_4\\
& (e_2) & -a_0 + 2a_1 - 2a_3 + a_4
 \end{array} }
\\ \\
\boxed{\begin{array}{lllll}
&J & (t_1t_3, t_1t_2, t_1^2, t_3^3, t_2t_4, t_2t_3, t_2^2)\\
&(e_0) & 4a_0 + 4a_4\\
&(e_1) & 3a_0 - a_1 - 3a_3 + 4a_4\\
&(e_2) & a_2 - 2a_3 + a_4 \end{array} }
\end{array}$$ 
\medskip

In this case $I^{top}=J^{top}= (t_1t_3, t_2, t_3^3, t_1^2)$ as expected by Remark \ref{toppo} and 
$J=J^{sat}\neq I^{sat}=(t_2, t_1t_3, t_1^2t_4, t_3^3, t_1^3)$ as proved in Proposition \ref{stessoh}

\end{document}